\documentclass[12pt]{amsart}

\usepackage[T1]{fontenc}
\usepackage{amsmath,amssymb,amsthm}
\usepackage{aliascnt,enumerate,shuffle}
\usepackage[margin=24mm]{geometry}
\usepackage{xcolor}
\definecolor{mylinkcolor}{RGB}{16, 156, 81}
\definecolor{mycitecolor}{RGB}{20, 80, 140}
\usepackage{hyperref}
\usepackage[nameinlink]{cleveref}
\hypersetup{
    setpagesize=false,
    bookmarksnumbered=true,
    bookmarksopen=true,
    colorlinks=true,
    linkcolor=mylinkcolor,
    citecolor=mycitecolor,
}

\numberwithin{equation}{section}
\allowdisplaybreaks[2]

\theoremstyle{plain}
\newtheorem{thm}{Theorem}[section]
\crefname{thm}{Theorem}{Theorems}

\newaliascnt{lem}{thm}
\newtheorem{lem}[lem]{Lemma}
\aliascntresetthe{lem}
\crefname{lem}{Lemma}{Lemmas}

\newaliascnt{prop}{thm}
\newtheorem{prop}[prop]{Proposition}
\aliascntresetthe{prop}
\crefname{prop}{Proposition}{Propositions}

\newaliascnt{cor}{thm}
\newtheorem{cor}[cor]{Corollary}
\aliascntresetthe{cor}
\crefname{cor}{Corollary}{Corollaries}

\newaliascnt{conj}{thm}
\newtheorem{conj}[conj]{Conjecture}
\aliascntresetthe{conj}
\crefname{conj}{Conjecture}{Conjectures}

\newaliascnt{dfn}{thm}
\newtheorem{dfn}[dfn]{Definition}
\aliascntresetthe{dfn}
\crefname{dfn}{Definition}{Definitions}

\newaliascnt{rem}{thm}
\newtheorem{rem}[rem]{Remark}
\aliascntresetthe{rem}
\crefname{rem}{Remark}{Remarks}

\newaliascnt{ex}{thm}
\newtheorem{ex}[ex]{Example}
\aliascntresetthe{ex}
\crefname{ex}{Example}{Examples}

\newtheorem{mainthm}{Main Theorem}

\crefname{mainthm}{Main Theorem}{Main Theorems}
\newaliascnt{mainconj}{mainthm}
\newtheorem{mainconj}[mainconj]{Main Conjecture}
\aliascntresetthe{mainconj}
\crefname{mainconj}{Main Conjecture}{Main Conjectures}

\crefname{section}{Section}{Sections}
\crefname{subsection}{Section}{Sections}

\newcommand{\odd}{\mathrm{odd}}
\newcommand{\even}{\mathrm{even}}
\newcommand{\QQ}{\mathbb{Q}}
\newcommand{\ZZ}{\mathbb{Z}}
\DeclareMathOperator{\Span}{span}
\DeclareMathOperator{\SL}{SL}
\DeclareMathOperator{\Fil}{Fil}
\DeclareMathOperator{\wt}{wt}
\DeclareMathOperator{\dep}{dep}

\newcommand{\h}{\mathfrak H}
\newcommand{\A}{\mathcal A}
\newcommand{\ZA}{{\mathcal Z_{\mathcal A}}}
\newcommand{\Zf}{{\mathcal Z^{\mathrm f}}}
\newcommand{\ZAf}{{\mathcal Z_{\mathcal A}^{\mathrm f}}}
\newcommand{\ZAw}[1]{{\mathcal Z_{\mathcal A,#1}}}
\newcommand{\ZAfw}[1]{{\mathcal Z_{\mathcal A,#1}^{\mathrm f}}}
\newcommand{\Zbarf}{{\overline{\mathcal Z}^{\mathrm f}}}
\newcommand{\za}[1]{\zeta_{\mathcal A}(#1)}
\newcommand{\zs}[1]{\zeta_{\mathcal S}(#1)}
\newcommand{\zf}[1]{\zeta^{\mathrm f}(#1)}
\newcommand{\zaf}[1]{\zeta_{\mathcal A}^{\mathrm f}(#1)}
\newcommand{\zsf}[1]{\zeta_{\mathcal S}^{\mathrm f}(#1)}
\newcommand{\zzf}[1]{\mathfrak z^{\mathrm f}(#1)}
\newcommand{\Dk}{\mathcal D_k}
\newcommand{\Dtk}{\widetilde{\mathcal D}_k}
\newcommand{\Dbk}{\overline{\mathcal D}_k}
\newcommand{\emptyword}{\mathbf 1}

\title{Formal finite multiple zeta values}

\author{Henrik Bachmann}
\address{Graduate School of Mathematics, Nagoya University, Nagoya, Japan.}
\email{henrik.bachmann@math.nagoya-u.ac.jp}

\author{Risan}
\address{Nagoya, Japan.}
\email{ptrrsn.1@gmail.com}

\date{}

\subjclass[2020]{Primary 11M32; Secondary 11F11}
\keywords{finite multiple zeta values, symmetric multiple zeta values, formal multiple zeta values, Kaneko-Zagier conjecture, period polynomials}

\begin{document}

\begin{abstract}
We study the algebra of formal finite multiple zeta values by imposing the stuffle and linear shuffle relations of Kaneko and Zagier. Our first results are a surjective homomorphism to the algebra of formal symmetric multiple zeta values and a parity reduction in depth at most four. In even weight we construct a quotient of the formal double zeta space which is isomorphic to the space of formal finite multiple zeta values of depth at most four. The proof combines an explicit depth-three reduction with a motivic lower bound. We also attach to every even period polynomial an explicit relation among the values $\zeta_{\mathcal A}(2a,1,2b,1)$, and these give all relations among the corresponding formal values.
\end{abstract}

\maketitle

\section{Introduction}

For an index $(k_1,\ldots,k_r) \in \ZZ_{\geq 1}^r$ with $k_1\geq2$, the \emph{multiple zeta value} is defined by
\begin{align}\label{eq:intro-mzv}
 \zeta(k_1,\ldots,k_r)
 =\sum_{m_1>\cdots>m_r>0}
 \frac{1}{m_1^{k_1}\cdots m_r^{k_r}}.
\end{align}
We write $\mathcal Z$ for the $\QQ$-algebra generated by these values. Multiple zeta values satisfy the stuffle and shuffle product relations. Their regularized combination gives the extended double shuffle relations of Ihara, Kaneko and Zagier~\cite{IKZ}, which are conjectured to generate all algebraic relations among multiple zeta values.

Kaneko and Zagier introduced two finite counterparts of \eqref{eq:intro-mzv}~\cite{KZ}. The first are the finite multiple zeta values
\begin{align}\label{eq:intro-fmzv}
 \za{k_1,\ldots,k_r}
 =\left(\sum_{p>m_1>\cdots>m_r>0}
 \frac{1}{m_1^{k_1}\cdots m_r^{k_r}}\bmod p\right)_p
 \in\A,
\end{align}
where
\[
 \A=\prod_{p\text{ prime}}\ZZ/p\ZZ\bigg/\bigoplus_{p\text{ prime}}\ZZ/p\ZZ.
\]
The second are the symmetric multiple zeta values $\zs{k_1,\ldots,k_r}\in\mathcal Z/\zeta(2)\mathcal Z$, defined as alternating sums of products of regularized multiple zeta values. Although the two definitions are very different, both families satisfy the stuffle product and the same linear shuffle relations. Kaneko and Zagier conjecture that these relations give all relations among finite multiple zeta values and that the assignment
\begin{align}\label{eq:intro-kz}
 \za{k_1,\ldots,k_r}\longmapsto\zs{k_1,\ldots,k_r}
\end{align}
induces an isomorphism $\ZA\simeq\mathcal Z/\zeta(2)\mathcal Z$. It is not known that \eqref{eq:intro-kz} is well defined.

Motivated by these conjectures, we define the algebra $\ZAf$ of \emph{formal finite multiple zeta values} by imposing precisely the stuffle and linear shuffle relations, and we denote the class of an index by $\zaf{k_1,\ldots,k_r}$. Starting from the algebra $\Zf$ of formal multiple zeta values, we also define formal symmetric multiple zeta values
\[
 \zsf{k_1,\ldots,k_r}\in
 \Zbarf:=\Zf/\zf{2}\Zf.
\]
Our first result is a formal version of the Kaneko--Zagier correspondence.

\begin{mainthm}\label{main:formal-kz}
There is a surjective homomorphism of graded $\QQ$-algebras
\[
 \begin{aligned}
 \varphi:\ZAf&\longrightarrow\Zbarf,\\
 \zaf{k_1,\ldots,k_r}&\longmapsto\zsf{k_1,\ldots,k_r}.
 \end{aligned}
\]
The classical finite and symmetric multiple zeta values are quotient realizations of $\ZAf$.
\end{mainthm}

The surjectivity is the formal version of Yasuda's theorem. Yasuda notes that his proof uses only regularized double shuffle relations and therefore applies to any realization satisfying them, see the introduction and the proof of the main theorem of~\cite{Y}. Applying that argument to the universal formal algebra $\Zf$, we obtain the result above. Global injectivity remains open and is the formal analogue of the Kaneko--Zagier conjecture. We prove injectivity on the even-weight part of depth at most four in \cref{cor:filtered-formal-kz-injective}.

Our second result concerns the depth filtration. The Kaneko--Zagier correspondence is expected to decrease depth by one. Together with the parity theorem for multiple zeta values~\cite{T,P}, this suggests that a formal finite value of weight $k$ and depth $d$ should reduce to lower depth when $k\equiv d\pmod2$. The corresponding statement is known for symmetric multiple zeta values in every depth~\cite[Corollary to Theorem~4]{KZ}, but it does not imply the result in the universal formal algebra or for finite multiple zeta values. Denote by $\Fil_d\ZAfw{k}$ the span of formal finite multiple zeta values of weight $k$ and depth $\leq d$, and by $\Fil_d\ZAw{k}$ the analogous span of finite multiple zeta values.

\begin{mainthm}\label{main:parity}
Let $k,d\geq1$ with $k\equiv d\pmod2$. If $d\leq4$, then
\[
 \Fil_d\ZAfw{k}\subseteq\Fil_{d-1}\ZAfw{k}.
\]
\end{mainthm}

The proof first reduces depth four to products of two odd-weight true single values \eqref{eq:true-single} modulo depth three and then eliminates these products by one generating-series form of the depth $1+3$ linear shuffle relation. It works entirely in $\ZAf$ and hence specializes simultaneously to finite and symmetric multiple zeta values.

The remaining part of the paper concerns period polynomials. For even $k\geq4$, let $W_k^{\even}$ be the space of homogeneous polynomials in $\QQ[X,Y]$ of degree $k-2$ satisfying
\begin{align}\label{eq:even-period-polynomials}
 p(X,Y)=p(X+Y,Y)-p(X+Y,X).
\end{align}
Every element is antisymmetric and contains only even powers of $X$ and $Y$, and
\[
 \dim_{\QQ}W_k^{\even}=\dim_{\QQ}\mathcal S_k+1,
\]
where $\mathcal S_k$ is the space of cusp forms of weight $k$ for $\SL_2(\ZZ)$ with rational Fourier coefficients. The line $\QQ(X^{k-2}-Y^{k-2})$ is the noncuspidal part. Gangl, Kaneko and Zagier showed that even period polynomials give relations among double zeta values~\cite{GKZ}, see also~\cite{BS,Ta} for further period-polynomial relations. Kaneko and Zagier conjecture that the finite multiple zeta values
\[
 \za{2a,1,2b,1},\qquad a,b\geq1,\quad 2a+2b+2=k,
\]
span the depth-four part in weight $k$ and satisfy exactly $\dim_{\QQ}\mathcal S_k$ independent relations~\cite{KZ}. On the symmetric side, Hara, Sakugawa and Tasaka recently expressed every cusp form with rational Fourier coefficients as a rational linear combination of symmetric triple Eisenstein series, whose constant terms give modular relations among symmetric triple zeta values~\cite{HST}.

To explain the relations among the values $\za{2a,1,2b,1}$, let $\Dk$ be the formal double zeta space of Gangl, Kaneko and Zagier and define
\[
 \Dtk=\Dk\big/\left(
 \QQ Z_k+\QQ Z_{1,k-1}
 +\Span_{\QQ}\{P_{r,s}:r,s\text{ even},\ r+s=k\}\right).
\]
Its dimension is
\[
 \dim_{\QQ}\Dtk=\frac{k}{2}-2-\dim_{\QQ}\mathcal S_k.
\]
For odd $r,s\geq3$ with $r+s=k$, put
\begin{align}\label{eq:H-source-element}
 H_{r,s}=\frac12\Bigg(&
 (rs-r+1)Z_{r,s}+(rs+s-1)Z_{s,r}
 +(s+1)Z_{s+1,r-1}-(r+1)Z_{r+1,s-1}\notag\\
 &+\sum_{\substack{3\leq\ell\leq k-3\\\ell\ \mathrm{odd}}}
 \left(\binom{k-\ell}{s}-\binom{k-\ell}{r}\right)
 P_{\ell,k-\ell}\Bigg)\in\Dtk.
\end{align}
These elements can be seen as the counterparts in $\Dtk$ of the values $\zeta_{\mathcal A}(r-1,1,s-1,1)$.

\begin{mainthm}\label{main:double-zeta-surjection}
Let $k\geq4$ be even.
\begin{enumerate}[(i)]
\item The elements $H_{r,s}$, where $r,s\geq3$ are odd and $r+s=k$, span $\Dtk$.
\item The rule
\[
 H_{r,s}\longmapsto\zaf{r-1,1,s-1,1}
\]
defines an isomorphism $\beta_k^{\mathrm f}:\Dtk\xrightarrow{\sim}\Fil_4\ZAfw{k}$.
\end{enumerate}
\end{mainthm}

The proof of part~(ii) has two steps. First, we derive an explicit depth-three reduction from the defining relations of $\ZAf$, which shows that $\beta_k^{\mathrm f}$ is well defined and surjective. Then a motivic realization gives the matching lower bound for the dimension of $\Fil_4\ZAfw{k}$. The coordinates realizing this map agree, after reversal, with depth-three elements introduced by Kina, who proved their shuffle relation and that they span the depth-three values of even weight~\cite{Ki}. We derive the complementary stuffle relation in $\ZAf$, which turns these coordinates into a realization of the reduced formal double zeta space, see \cref{rem:kina-comparison}.

\begin{mainthm}\label{main:classical-beta}
For every even $k\geq4$, 
\[
 H_{r,s}\longmapsto\za{r-1,1,s-1,1}
\]
defines a surjective map $\beta_k:\Dtk\to\Fil_4\ZAw{k}$. Consequently,
\[
 \Fil_4\ZAw{k}
 =\Span_{\QQ}\{\za{r-1,1,s-1,1}:r,s\geq3\text{ odd},\ r+s=k\},
\]
and these values satisfy at least $\dim_{\QQ}\mathcal S_k$ independent linear relations.
\end{mainthm}

We expect that also $\beta_k$ is an isomorphism, but for finite multiple zeta values the matching lower bound is open.

\begin{mainconj}\label{main:double-zeta}
For every even $k\geq4$, the map $\beta_k$ is an isomorphism.
\end{mainconj}

The period-polynomial theorem of~\cite{GKZ} gives all relations among the classes $Z_{r,k-r}$ with $r,k-r\geq3$ odd in $\Dtk$. Using this, we give in \cref{thm:recursive-period-transfer} an explicit linear isomorphism
\[
 W_k^{\even}/\QQ(X^{k-2}-Y^{k-2})
 \longrightarrow
 \left\{(c_r)_{\substack{3\leq r\leq k-3\\r\ \odd}}:
 \sum_{\substack{3\leq r\leq k-3\\r\ \odd}}
 c_r\zaf{r-1,1,k-r-1,1}=0\right\}.
\]
Starting from the coefficients of a period polynomial, the construction only uses a triangular recurrence, differentiation, coefficient extraction and the change of variables of Gangl, Kaneko and Zagier~\cite{GKZ}. In particular, every even period polynomial gives an explicit relation among finite multiple zeta values, and we obtain $\dim_{\QQ}\mathcal S_k$ independent relations in this way. For example, for the period polynomial $X^2Y^2(X^2-Y^2)^3$ of weight $12$, our construction gives the relation
\begin{align}\label{eq:intro-weight-twelve-relation}
 16\za{2,1,8,1}+9\za{4,1,6,1}
 +18\za{6,1,4,1}-2\za{8,1,2,1}=0.
\end{align}
By \cref{main:double-zeta-surjection}, this is, up to a scalar, the only relation among the formal values $\zaf{2a,1,2b,1}$ of weight $12$. If $\beta_k$ is bijective, these are all relations among the values $\za{2a,1,2b,1}$, and the depth-four conjecture of Kaneko and Zagier follows.

The organization of this paper is as follows. In \Cref{sec:prelim} we recall finite, symmetric and formal multiple zeta values. In \Cref{sec:formal-finite} we define $\ZAf$ and prove \cref{main:formal-kz}. The parity theorem is proved in \Cref{sec:parity}. In \Cref{sec:double-zeta} we study $\Dtk$, construct the maps $\beta_k^{\mathrm f}$ and $\beta_k$, and derive the period-polynomial relations.

\Cref{main:formal-kz,main:parity} are contained in the master's thesis \cite{R} of the second author, written under the supervision of the first author.

In an earlier version of this paper, the identity in \cref{thm:depth-three-reduction} was stated as a conjecture, which we had verified by computer for every even weight $k\leq18$. Part~(ii) of \cref{main:double-zeta-surjection} was therefore only proved under this assumption and only as a surjection, and it was conjectured that $\beta_k^{\mathrm f}$ is an isomorphism.

\subsection*{Acknowledgements}
We thank Masanobu Kaneko for fruitful discussions on this topic. This project was partially supported by JSPS KAKENHI Grants 23K03030 and 26K22254.

\section{Finite, symmetric and formal multiple zeta values}\label{sec:prelim}

Throughout the paper an \emph{index} is a tuple $\boldsymbol{k}=(k_1,\ldots,k_r)$ of positive integers. Its weight and depth are
\[
 \wt(\boldsymbol{k})=k_1+\cdots+k_r,
 \qquad \dep(\boldsymbol{k})=r.
\]
We use decreasing summation indices as in \eqref{eq:intro-fmzv}. Note that Kaneko and Zagier use increasing summation indices in~\cite{KZ}, and the two conventions are related by reversal.

\subsection{The two products}

Let $\h=\QQ\langle x,y\rangle$ and
\[
 \h^0=\QQ+x\h y,
 \qquad
 \h^1=\QQ+\h y.
\]
We put $z_k=x^{k-1}y$ for $k\geq1$, i.e. the words in $\h^1$ are $\emptyword$ and $z_{k_1}\cdots z_{k_r}$. The \emph{stuffle product} is the bilinear product on $\h^1$ determined by
\begin{align}\label{eq:stuffle}
 \emptyword*w=w*\emptyword=w,
 \qquad
 z_aw*z_bv=z_a(w*z_bv)+z_b(z_aw*v)+z_{a+b}(w*v).
\end{align}
The \emph{shuffle product} on $\h$ is determined by
\begin{align}\label{eq:shuffle}
 \emptyword\shuffle w=w\shuffle\emptyword=w,
 \qquad
 au\shuffle bv=a(u\shuffle bv)+b(au\shuffle v)
\end{align}
for $a,b\in\{x,y\}$. We denote the corresponding commutative algebras by $\h^1_*$ and $\h^1_\shuffle$.

Let
\[
 R(z_{k_1}\cdots z_{k_r})=z_{k_r}\cdots z_{k_1}
\]
be reversal. The deconcatenation coproduct
\[
 \Delta(w)=\sum_{uv=w}u\otimes v
\]
makes $\h^1_*$ a Hopf algebra~\cite{HI}. If $\Sigma$ replaces any collection of separators in a reversed word by contractions $z_a\diamond z_b=z_{a+b}$ and $\epsilon_{\dep}(w)=(-1)^{\dep(w)}w$, its antipode is
\begin{align}\label{eq:antipode}
 S_*=\Sigma\circ\epsilon_{\dep}\circ R.
\end{align}
We will use the antipode identity
\begin{align}\label{eq:antipode-identity}
 \sum_{uv=w}S_*(u)*v=0
\end{align}
for every nonempty word $w$.

\subsection{Finite and symmetric values}

The map $Z_{\mathcal A}:\h^1\to\A$ defined by \eqref{eq:intro-fmzv} is an algebra homomorphism for the stuffle product. It also satisfies the linear shuffle relation of Kaneko and Zagier: for all words $w,v\in\h^1$,
\begin{align}\label{eq:finite-linear-shuffle}
 Z_{\mathcal A}(w\shuffle v)=(-1)^{\wt(w)}Z_{\mathcal A}(R(w)v).
\end{align}
In particular,
\begin{align}\label{eq:finite-reversal}
 \za{k_r,\ldots,k_1}=(-1)^{k_1+\cdots+k_r}\za{k_1,\ldots,k_r}.
\end{align}
We write $\ZA\subset\A$ for the $\QQ$-algebra spanned by all finite multiple zeta values.

To recall the second finite version, let $\zeta^{*,T}$ and $\zeta^{\shuffle,T}$ denote the stuffle and shuffle regularizations of multiple zeta values. For $\bullet\in\{*,\shuffle\}$ set
\begin{align}\label{eq:symmetric-classical}
 \zeta_{\mathcal S}^{\bullet}(k_1,\ldots,k_r)
 =\sum_{j=0}^r(-1)^{k_1+\cdots+k_j}
 \zeta^{\bullet,T}(k_j,\ldots,k_1)
 \zeta^{\bullet,T}(k_{j+1},\ldots,k_r),
\end{align}
where the value attached to the empty index is~$1$. This expression is independent of $T$, and the two choices of regularization are congruent modulo $\zeta(2)\mathcal Z$. Their common class is $\zs{k_1,\ldots,k_r}$. Symmetric multiple zeta values satisfy the stuffle product and \eqref{eq:finite-linear-shuffle}, see~\cite{KZ,K}.

\subsection{Formal multiple zeta values}

The stuffle regularization map $\operatorname{reg}_*: \h^1_*\simeq\h^0_*[z_1]\to\h^0_*$ is the identity on $\h^0$ and sends $z_1$ to zero. The algebra of \emph{formal multiple zeta values} is
\begin{align}\label{eq:formal-mzv}
 \Zf=\h^0_*\big/\mathrm{EDS}_*,
\end{align}
where $\mathrm{EDS}_*$ is the ideal generated by
\[
 \operatorname{reg}_*(w*v-w\shuffle v),
 \qquad w\in\h^0,\quad v\in\h^1.
\]
For an admissible index we denote the class of $z_{k_1}\cdots z_{k_r}$ by $\zf{k_1,\ldots,k_r}$. There is a natural surjection $\Zf\to\mathcal Z$, and the extended double shuffle conjecture states that it is an isomorphism.

The two polynomial regularizations are available inside the formal algebra itself. Namely, there are unique homomorphisms
\[
 Z^{\mathrm f,*}:\h^1_*\longrightarrow\Zf[T],
 \qquad
 Z^{\mathrm f,\shuffle}:\h^1_\shuffle\longrightarrow\Zf[T]
\]
which extend the quotient map on $\h^0$ and send $z_1$ to~$T$. Put
\begin{align}\label{eq:formal-A}
 A^{\mathrm f}(u)=\exp\left(\sum_{n\geq2}\frac{(-1)^n}{n}\zf{n}u^n\right),
\end{align}
and define the $\Zf$-linear map $\rho^{\mathrm f}:\Zf[T]\to\Zf[T]$ coefficientwise by
\begin{align}\label{eq:formal-rho}
 \rho^{\mathrm f}(e^{Tu})=A^{\mathrm f}(u)e^{Tu}.
\end{align}

\begin{prop}\label{prop:formal-regularization}
One has
\begin{align}\label{eq:formal-reg-theorem}
 Z^{\mathrm f,\shuffle}=\rho^{\mathrm f}\circ Z^{\mathrm f,*}.
\end{align}
Moreover, for every $m\geq1$,
\begin{align}\label{eq:formal-euler}
 \zf{2m}=-\frac{B_{2m}}{2(2m)!}\bigl(-24\zf{2}\bigr)^m.
\end{align}
In particular, $\zf{2m}\in\zf{2}\Zf$.
\end{prop}

\begin{proof}
The proof of the regularization theorem in~\cite[Theorem~2]{IKZ} is an identity in the two polynomial algebras $\h^0_*[z_1]$ and $\h^0_\shuffle[z_1]$. More precisely, after taking the coefficient of $u^n$ in \eqref{eq:formal-rho}, the difference
\[
 Z^{\mathrm f,\shuffle}(w)-\rho^{\mathrm f}
 \bigl(Z^{\mathrm f,*}(w)\bigr)
\]
is the image of a finite $\QQ$-linear combination of the regularized elements $\operatorname{reg}_*(a*b-a\shuffle b)$ which generate $\mathrm{EDS}_*$. Therefore it is zero in the quotient \eqref{eq:formal-mzv}. This proves \eqref{eq:formal-reg-theorem} without using a numerical realization. Formula \eqref{eq:formal-euler} is likewise a consequence of the extended double shuffle relations, and the formal calculation is given in~\cite[Section~2.2]{BIM}. The last assertion is immediate.
\end{proof}

\section{Formal finite and formal symmetric multiple zeta values}\label{sec:formal-finite}

We now introduce the universal algebra which satisfies the two families of relations common to finite and symmetric multiple zeta values. This algebra is already considered by Kaneko and Zagier~\cite{KZ}, whose conjecture states that its defining relations give all relations among finite multiple zeta values. It is studied in detail in the master's thesis~\cite{R}, and a closely related formulation appears in~\cite[Section~7]{A}.

\begin{dfn}\label{def:formal-finite}
Let $\mathfrak L\subset\h^1_*$ be the ideal generated by the elements
\begin{align}\label{eq:formal-linear-shuffle}
 w\shuffle v-(-1)^{\wt(w)}R(w)v
\end{align}
for words $w,v\in\h^1$. The algebra of \emph{formal finite multiple zeta values} is
\[
 \ZAf=\h^1_*/\mathfrak L.
\]
We write $\zaf{k_1,\ldots,k_r}$ for the class of $z_{k_1}\cdots z_{k_r}$ and set $\zaf{\varnothing}=1$.
\end{dfn}

The algebra $\ZAf$ is graded by weight, and we denote its weight-$k$ part by $\ZAfw{k}$. By construction, the stuffle product and the linear shuffle relation hold in $\ZAf$. Specializing \eqref{eq:formal-linear-shuffle} at the empty word gives the reversal relation
\begin{align}\label{eq:formal-reversal}
 \zaf{k_r,\ldots,k_1}=(-1)^{k_1+\cdots+k_r}\zaf{k_1,\ldots,k_r}.
\end{align}

\begin{prop}\label{prop:realizations}
The classical finite and symmetric multiple zeta values define surjective algebra homomorphisms
\begin{align*}
 \pi_{\mathcal A}:\ZAf&\longrightarrow\ZA,
 &\pi_{\mathcal S}:\ZAf&\longrightarrow
      \mathcal Z/\zeta(2)\mathcal Z,\\
 \zaf{\boldsymbol{k}}&\longmapsto\za{\boldsymbol{k}},
 &\zaf{\boldsymbol{k}}&\longmapsto\zs{\boldsymbol{k}}.
\end{align*}
\end{prop}

\begin{proof}
The stuffle product and \eqref{eq:formal-linear-shuffle} hold for both realizations. The first map is surjective by the definition of $\ZA$. Surjectivity of the second is Yasuda's theorem~\cite{Y}.
\end{proof}

Consequently, every identity in formal finite multiple zeta values proved below also holds for classical finite and classical symmetric multiple zeta values.

We will use the following low-depth structure repeatedly in the parity argument.

\begin{prop}\label{prop:low-depth}
For positive integers $a,b$ one has
\begin{enumerate}[(i)]
 \item $\zaf{a}=0$,
 \item if $a+b$ is even, then $\zaf{a,b}=0$,
 \item on setting
 \begin{align}\label{eq:true-single}
  \zzf{1}=0,\qquad \zzf{k}=\frac1k\zaf{k-1,1}\qquad(k\geq2),
 \end{align}
 one has
 \begin{align}\label{eq:formal-double-evaluation}
  \zaf{a,b}=(-1)^a\binom{a+b}{a}\zzf{a+b}.
 \end{align}
\end{enumerate}
\end{prop}

\begin{proof}
We begin with depth one and argue by induction on the weight. The case $k=1$ follows from reversal. For $k=2$, the linear shuffle and stuffle relations for $z_1$ and $z_1$ give, respectively,
\[
 3\zaf{1,1}=0,
 \qquad
 0=\zaf{1}^2=2\zaf{1,1}+\zaf{2},
\]
and hence $\zaf{2}=0$. Let now $k\geq3$. The linear shuffle relation with $w=z_1$ and $v=z_{k-1}$ gives
\begin{align}\label{eq:depth-one-shuffle}
 0=2\zaf{k-1,1}+\sum_{i=2}^{k-2}\zaf{k-i,i}+2\zaf{1,k-1}.
\end{align}
By the induction hypothesis and stuffle,
\[
 \zaf{k-i,i}+\zaf{i,k-i}=-\zaf{k}.
\]
Pairing the terms in \eqref{eq:depth-one-shuffle}, including the middle term separately when $k$ is even, shows that its right-hand side is $-(k+1)\zaf{k}/2$. Thus $\zaf{k}=0$.

We will also use the formal sum formula
\begin{align}\label{eq:formal-sum}
 \sum_{\substack{k_1+\cdots+k_d=k\\k_i\geq1}}
 \zaf{k_1,\ldots,k_d}=0.
\end{align}
Indeed, Hoffman's symmetric-sum identity writes the left-hand side as a polynomial, for the stuffle product, in depth-one values, and these have just been shown to vanish~\cite{H1}. Under the finite realization, \eqref{eq:formal-sum} recovers a special case of the sum formula of Saito--Wakabayashi~\cite{SW}.

If $a+b$ is even, depth-one vanishing, stuffle and reversal give
\[
 0=\zaf{a,b}+\zaf{b,a}=2\zaf{a,b},
\]
which proves part~(ii).

For part~(iii), it is enough to establish the recurrence
\begin{align}\label{eq:double-recurrence}
 (a+1)\zaf{a+1,b}+(b+1)\zaf{a,b+1}=0.
\end{align}
It is immediate from part~(ii) when $a+b+1$ is even. Suppose that this weight is odd. The linear shuffle relation for $z_1$ and $z_az_b$ is
\begin{align}\label{eq:double-linear-shuffle}
0={}&\zaf{1,a,b}
+\sum_{\substack{u+v=a+1\\u,v\geq1}}\zaf{u,v,b}
+\sum_{\substack{u+v=b+1\\u,v\geq1}}\zaf{a,u,v}
+\zaf{a,b,1}.
\end{align}
In addition, expand the following three zero products:
\[
 \zaf{1}\zaf{a,b},\qquad
 \sum_{\substack{u+v=a+1\\u,v\geq1}}\zaf{u}\zaf{v,b},\qquad
 \sum_{\substack{u+v=b+1\\u,v\geq1}}\zaf{u}\zaf{a,v}.
\]
Their stuffle expansions sum to
\begin{align*}
 0={}&\zaf{1,a,b}+\zaf{a,1,b}+\zaf{a,b,1}
 +\zaf{a+1,b}+\zaf{a,b+1}\\
 &+\sum_{u+v=a+1}\bigl(\zaf{u,v,b}+\zaf{v,u,b}+\zaf{v,b,u}\bigr)
 +a\,\zaf{a+1,b}+\sum_{u+v=a+1}\zaf{v,u+b}\\
 &+\sum_{u+v=b+1}\bigl(\zaf{u,a,v}+\zaf{a,u,v}+\zaf{a,v,u}\bigr)
 +\sum_{u+v=b+1}\zaf{u+a,v}+b\,\zaf{a,b+1},
\end{align*}
with all summation variables positive. Subtract twice \eqref{eq:double-linear-shuffle}. In the difference, the sums $\sum\zaf{u,v,b}$ and $\sum\zaf{a,u,v}$ cancel after relabeling, the depth-two contractions combine into $(a+1)\zaf{a+1,b}+(b+1)\zaf{a,b+1}$ together with the complete sum $\sum_{u+v=a+b+1}\zaf{u,v}$, which vanishes by \eqref{eq:formal-sum}, and the leftover sums $\sum\zaf{v,b,u}$ and $\sum\zaf{u,a,v}$ cancel in pairs by reversal in odd weight. This leaves
\begin{align}\label{eq:recurrence-remainder}
 0={}&(a+1)\zaf{a+1,b}+(b+1)\zaf{a,b+1}\\
 &+\zaf{a,1,b}-\zaf{1,a,b}-\zaf{a,b,1}.\notag
\end{align}
The last line vanishes. To see this, set
\[
 T=\zaf{b,a,1}+\zaf{1,b,a}+\zaf{a,1,b}.
\]
Expanding stuffle products with one depth-one factor, we get
\[
 0=2T+\bigl(\zaf{1,a,b}+\zaf{a,b,1}+\zaf{b,1,a}\bigr)
   =\bigl(2-(-1)^{a+b}\bigr)T,
\]
where the last equality is reversal. Hence $T=0$, and in odd total weight this is exactly the last line of \eqref{eq:recurrence-remainder}. This proves \eqref{eq:double-recurrence}. Iterating the recurrence from $\zaf{k-1,1}=k\zzf{k}$ gives \eqref{eq:formal-double-evaluation}.
\end{proof}

The element $\zzf{k}$ can be seen as the formal finite replacement for a single zeta value. Under the finite realization it maps to the Bernoulli class $(B_{p-k}/k)_p$, and it vanishes in even weight.

\subsection{Formal symmetric multiple zeta values}

Let $\epsilon_{\wt}(w)=(-1)^{\wt(w)}w$. If $f,g:\h^1\to B$ are linear maps to a commutative algebra, their deconcatenation convolution is
\[
 (f\star g)(w)=\sum_{uv=w}f(u)g(v).
\]
For $\bullet\in\{*,\shuffle\}$ define
\begin{align}\label{eq:formal-symmetric-map}
 Z_{\mathcal S}^{\mathrm f,\bullet}
 =\bigl(Z^{\mathrm f,\bullet}\circ R\circ\epsilon_{\wt}\bigr)\star Z^{\mathrm f,\bullet}.
\end{align}
Explicitly,
\begin{align}\label{eq:formal-symmetric-explicit}
 Z_{\mathcal S}^{\mathrm f,\bullet}(z_{k_1}\cdots z_{k_r})
 =\sum_{j=0}^r(-1)^{k_1+\cdots+k_j}
 \zeta^{\mathrm f,\bullet,T}(k_j,\ldots,k_1)
 \zeta^{\mathrm f,\bullet,T}(k_{j+1},\ldots,k_r).
\end{align}

We prove the properties of \eqref{eq:formal-symmetric-map} inside the formal algebra. We begin with the product which does not require passing modulo $\zf{2}$.

\begin{lem}\label{lem:formal-symmetric-stuffle}
The map $Z_{\mathcal S}^{\mathrm f,*}:\h^1_*\to\Zf[T]$ is an algebra homomorphism.
\end{lem}

\begin{proof}
The maps $\epsilon_{\wt}$ and $R$ are algebra automorphisms of the stuffle algebra: the first assertion follows from additivity of the weight, and the second follows directly by reversing the recursion \eqref{eq:stuffle}. Hence $Z^{\mathrm f,*}\circ R\circ\epsilon_{\wt}$ and $Z^{\mathrm f,*}$ are stuffle homomorphisms. The deconcatenation coproduct is compatible with the stuffle product,
\[
 \Delta(w*v)=\Delta(w)*\Delta(v),
\]
where the product on $\h^1\otimes\h^1$ is componentwise. Consequently, the convolution of two stuffle homomorphisms is again a stuffle homomorphism. Applying this observation to \eqref{eq:formal-symmetric-map} proves the lemma.
\end{proof}

For an admissible index $\boldsymbol{k}$, with the empty index allowed, define the leading-one series
\begin{align}\label{eq:leading-one-series}
 G_T^\bullet(\boldsymbol{k};X)
 =\sum_{j\geq0}Z^{\mathrm f,\bullet}(z_1^jz_{\boldsymbol{k}})X^j
 \in\Zf[T][[X]].
\end{align}

\begin{lem}\label{lem:leading-one}
For $\bullet\in\{*,\shuffle\}$ one has
\begin{align}
 G_T^\bullet(\boldsymbol{k};X)
 &=e^{XT}G_0^\bullet(\boldsymbol{k};X),\label{eq:leading-one-T}\\
 G_0^\shuffle(\boldsymbol{k};X)
 &=A^{\mathrm f}(X)G_0^*(\boldsymbol{k};X).\label{eq:leading-one-comparison}
\end{align}
\end{lem}

\begin{proof}
The polynomial decomposition $\h^1_\bullet=\h^0_\bullet[z_1]$ gives the translation identity
\[
 G_{T+U}^\bullet(\boldsymbol{k};X)
 =e^{UX}G_T^\bullet(\boldsymbol{k};X).
\]
Indeed, after comparing coefficients of $U^j$, this is precisely the regularization identity obtained by extracting $j$ leading copies of $z_1$. Equivalently, it follows by induction from the defining shuffle or stuffle recursion. Taking $T=0$ proves \eqref{eq:leading-one-T}. On the other hand, \cref{prop:formal-regularization} and \eqref{eq:leading-one-T} give
\[
 G_T^\shuffle(\boldsymbol{k};X)
 =\rho^{\mathrm f}\bigl(e^{XT}G_0^*(\boldsymbol{k};X)\bigr)
 =A^{\mathrm f}(X)e^{XT}G_0^*(\boldsymbol{k};X).
\]
Setting $T=0$ proves \eqref{eq:leading-one-comparison}.
\end{proof}

\begin{prop}\label{prop:formal-symmetric-regularization}
For every $w\in\h^1$, both $Z_{\mathcal S}^{\mathrm f,*}(w)$ and $Z_{\mathcal S}^{\mathrm f,\shuffle}(w)$ are independent of $T$ and therefore belong to $\Zf$. Moreover,
\begin{align}\label{eq:formal-reg-congruence}
 Z_{\mathcal S}^{\mathrm f,*}(w)
 \equiv Z_{\mathcal S}^{\mathrm f,\shuffle}(w)
 \pmod{\zf{2}\Zf}.
\end{align}
\end{prop}

\begin{proof}
We group the summands in \eqref{eq:formal-symmetric-explicit} according to the maximal block of entries equal to~$1$ through which the cut passes. Write such a portion of the word as
\[
 z_{r_m}\cdots z_{r_1}z_1^h z_{s_1}\cdots z_{s_n},
\]
where $z_{\boldsymbol r}=z_{r_1}\cdots z_{r_m}$ and $z_{\boldsymbol s}=z_{s_1}\cdots z_{s_n}$ are admissible, and either of them may be empty. The sum of the contributions from the $h+1$ cuts in this block is the coefficient of $X^h$ in
\begin{align}\label{eq:block-series}
 P^\bullet(\boldsymbol r,\boldsymbol s;X)
 =(-1)^{\wt(\boldsymbol r)}
 G_T^\bullet(\boldsymbol r;-X)G_T^\bullet(\boldsymbol s;X).
\end{align}
By \eqref{eq:leading-one-T}, the two exponential factors in \eqref{eq:block-series} are $e^{-XT}$ and $e^{XT}$, and hence cancel. Thus every group of summands, and therefore the full expression \eqref{eq:formal-symmetric-explicit}, is independent of~$T$.

For the comparison of the two regularizations, \eqref{eq:leading-one-comparison} gives
\begin{align}\label{eq:block-comparison}
 P^\shuffle(\boldsymbol r,\boldsymbol s;X)
 =A^{\mathrm f}(-X)A^{\mathrm f}(X)
 P^*(\boldsymbol r,\boldsymbol s;X).
\end{align}
The odd terms in the logarithm cancel, so that
\begin{align}\label{eq:A-even-factor}
 A^{\mathrm f}(-X)A^{\mathrm f}(X)
 =\exp\left(\sum_{n\geq1}\frac{\zf{2n}}{n}X^{2n}\right)
 \equiv1\pmod{\zf{2}\Zf[[X]]}
\end{align}
by \eqref{eq:formal-euler}. Comparing coefficients of $X^h$ in \eqref{eq:block-comparison} and summing over the blocks proves \eqref{eq:formal-reg-congruence}.
\end{proof}

The remaining relation is combinatorial. We give the operator argument in detail, since this is the point where the linear shuffle relation is transferred to the formal symmetric values.

\begin{lem}\label{lem:cut-shuffle}
Let $B$ be a commutative $\QQ$-algebra and let $Z:\h^1_\shuffle\to B$ be an algebra homomorphism. Define
\[
 Z_{\mathcal S,Z}=(Z\circ R\circ\epsilon_{\wt})\star Z.
\]
Then, for all words $w,v\in\h^1$,
\begin{align}\label{eq:abstract-linear-shuffle}
 Z_{\mathcal S,Z}(w\shuffle v)
 =(-1)^{\wt(w)}Z_{\mathcal S,Z}(R(w)v).
\end{align}
\end{lem}

\begin{proof}
For a word $z_{k_1}\cdots z_{k_r}$, let
\[
 S_i(z_{k_1}\cdots z_{k_r})
 =(-1)^{k_1+\cdots+k_i}
 (z_{k_i}\cdots z_{k_1})\shuffle
 (z_{k_{i+1}}\cdots z_{k_r}),
 \qquad S=\sum_{i=0}^rS_i.
\]
Since $Z$ is a shuffle homomorphism, the definition of the convolution gives $Z_{\mathcal S,Z}=Z\circ S$. Under the reversal dictionary $z_{k_1}\cdots z_{k_r}\mapsto(k_r,\ldots,k_1)$ between our convention and the increasing-index convention of~\cite{KZ}, and after the relabeling $i\mapsto r-i$, the operators $S_i$ correspond to the cut operators of Kaneko and Zagier, and \cite[Proposition~10]{KZ} gives
\[
 S\circ S_i=S\qquad(0\leq i\leq r).
\]
If $d=\dep(w)$, then $S_d(R(w)v)=(-1)^{\wt(w)}w\shuffle v$. Hence $S(w\shuffle v)=(-1)^{\wt(w)}S(R(w)v)$, and applying $Z$ proves the claim.
\end{proof}

\begin{prop}\label{prop:formal-symmetric-properties}
The common class $Z_{\mathcal S}^{\mathrm f}(w)\in\Zbarf$ of the two expressions in \cref{prop:formal-symmetric-regularization} satisfies, for all words $w,v\in\h^1$,
\begin{align}\label{eq:formal-symmetric-relations}
 Z_{\mathcal S}^{\mathrm f}(w*v)&=Z_{\mathcal S}^{\mathrm f}(w)Z_{\mathcal S}^{\mathrm f}(v),\\
 Z_{\mathcal S}^{\mathrm f}(w\shuffle v)&=(-1)^{\wt(w)}Z_{\mathcal S}^{\mathrm f}(R(w)v).\notag
\end{align}
\end{prop}

\begin{proof}
The first identity holds already for $Z_{\mathcal S}^{\mathrm f,*}$ by \cref{lem:formal-symmetric-stuffle}. For the second, use the shuffle representative $Z_{\mathcal S}^{\mathrm f,\shuffle}$ and apply \cref{lem:cut-shuffle} to $Z=Z^{\mathrm f,\shuffle}$. The two representatives define the same class by \eqref{eq:formal-reg-congruence}.
\end{proof}

\begin{dfn}\label{def:formal-symmetric}
For an index $\boldsymbol{k}=(k_1,\ldots,k_r)$, its \emph{formal symmetric multiple zeta value} is
\[
 \zsf{k_1,\ldots,k_r}
 =Z_{\mathcal S}^{\mathrm f}(z_{k_1}\cdots z_{k_r})\in\Zbarf.
\]
\end{dfn}

The following depth-two evaluation will be used as a check on the definition and to identify the true single values.

\begin{prop}\label{prop:formal-symmetric-depth-two}
For $a,b\geq1$ and $k=a+b$, one has
\begin{align}\label{eq:formal-symmetric-depth-two}
 \zsf{a,b}=
 \begin{cases}
  0,&k\text{ even},\\
  (-1)^a\binom{k}{a}\zf{k},&k\text{ odd}
 \end{cases}
 \qquad\text{in }\Zbarf.
\end{align}
\end{prop}

\begin{proof}
Set $T=0$, which is allowed by \cref{prop:formal-symmetric-regularization}. From \eqref{eq:formal-symmetric-explicit},
\begin{align}\label{eq:formal-symmetric-depth-two-start}
 \zsf{a,b}
 ={}&\zeta^{\mathrm f,*}(a,b)
 +(-1)^a\zeta^{\mathrm f,*}(a)\zeta^{\mathrm f,*}(b)
 +(-1)^k\zeta^{\mathrm f,*}(b,a)
 \pmod{\zf{2}\Zf}.
\end{align}
If $k$ is even, the stuffle relation turns the right-hand side into
\[
 \bigl(1+(-1)^a\bigr)\zeta^{\mathrm f,*}(a)
 \zeta^{\mathrm f,*}(b)-\zf{k}.
\]
This belongs to $\zf{2}\Zf$: if $a$ is odd the product has coefficient zero, while if $a$ is even both factors have even weight and \eqref{eq:formal-euler} applies.

Suppose that $k$ is odd. The depth-two part of the extended double shuffle relation gives
\begin{align}\label{eq:formal-double-zeta-odd}
 \zeta^{\mathrm f,*}(a,b)\big|_{T=0}
 \equiv\frac12\left((-1)^a\binom{k}{a}-1\right)\zf{k}
 \pmod{\zf{2}\Zf}.
\end{align}
More precisely, for $a\geq2$, the Euler decomposition obtained from the regularized double shuffle relations is
\begin{align*}
 \zf{a,b}={}&(-1)^b
 \sum_{\substack{2\leq j\leq k-1\\j\,\mathrm{even}}}
 \left\{\binom{k-j-1}{a-1}+\binom{k-j-1}{b-1}
       +(-1)^b\delta_{j,a}\right\}\zf{j}\zf{k-j}\\
 &+\frac12\left((-1)^a\binom{k}{a}-1\right)\zf{k}.
\end{align*}
This identity holds in $\Zf$ because it is derived before evaluation from the generators of $\mathrm{EDS}_*$, see also~\cite[Section~2.2]{BIM}. Every term in the sum is divisible by $\zf{2}$ by \eqref{eq:formal-euler}, which proves \eqref{eq:formal-double-zeta-odd} for $a\geq2$. If $a=1$, the stuffle relation at $T=0$,
\[
 0=\zeta^{\mathrm f,*}(1,b)+\zeta^{\mathrm f,*}(b,1)+\zf{k},
\]
reduces the assertion to the already proved case $(b,1)$. Thus \eqref{eq:formal-double-zeta-odd} also covers all regularized boundary cases. Substituting it and its version with $a,b$ interchanged into \eqref{eq:formal-symmetric-depth-two-start} gives $(-1)^a\binom{k}{a}\zf{k}$, as claimed.
\end{proof}

We can now prove the first main theorem.

\begin{thm}\label{thm:formal-kz}
The map
\begin{align*}
 \varphi:\ZAf&\longrightarrow\Zbarf,\\
 \zaf{\boldsymbol{k}}&\longmapsto\zsf{\boldsymbol{k}}
\end{align*}
is a well-defined surjective homomorphism of graded $\QQ$-algebras.
\end{thm}

\begin{proof}
By \cref{prop:formal-symmetric-properties}, the map $Z_{\mathcal S}^{\mathrm f}$ is a stuffle homomorphism and annihilates every generator \eqref{eq:formal-linear-shuffle} of $\mathfrak L$. It therefore factors through $\ZAf$.

For surjectivity, Yasuda's proof that symmetric multiple zeta values generate $\mathcal Z/\zeta(2)\mathcal Z$ uses only regularized double shuffle relations, as he notes in the introduction of~\cite{Y}. The same proof therefore applies to their universal quotient $\Zf$ and shows that the formal symmetric values span $\Zbarf$.
\end{proof}

\begin{cor}\label{cor:true-single}
For every $k\geq2$,
\[
 \varphi(\zzf{k})=\zf{k}\pmod{\zf{2}\Zf}.
\]
\end{cor}

\begin{proof}
By \cref{prop:formal-symmetric-depth-two}, $\zsf{k-1,1}=k\zf{k}$ in $\Zbarf$: for odd $k$ this is the second case of \eqref{eq:formal-symmetric-depth-two}, while for even $k$ both sides vanish by \eqref{eq:formal-euler}. Divide by~$k$ and use \eqref{eq:true-single}.
\end{proof}

\begin{conj}\label{conj:formal-kz}
The homomorphism $\varphi$ is an isomorphism.
\end{conj}

This conjecture also appears in~\cite[Conjecture~7.6]{A}. To locate the difficulty in this conjecture, note that on the vector space of indices, the cut-shuffle operator underlying the symmetric construction has kernel exactly the linear span of the finite linear shuffle expressions~\cite[Proposition~12]{KZ}. The definition of $\ZAf$, however, quotients by the \emph{stuffle ideal} generated by this span. Therefore the unresolved point is not another regularization issue, but the interaction between multiplication and the cut-shuffle kernel, followed by the formal multiple-zeta quotient. In particular, a multiplicatively compatible splitting of the cut-shuffle map would prove \cref{conj:formal-kz}, whereas the kernel statement alone does not provide such a splitting.

\section{Parity}\label{sec:parity}

For $d\geq0$, let
\begin{align}\label{eq:depth-filtration}
 \Fil_d\ZAfw{k}
 =\Span_{\QQ}\{\zaf{k_1,\ldots,k_r}:r\leq d,\ k_1+\cdots+k_r=k\},
\end{align}
and let $\Fil_d\ZAw{k}\subset\ZA$ be defined in the same way with $\zaf{\cdot}$ replaced by $\za{\cdot}$. The product respects weight and depth filtrations. The antipode gives a first parity reduction which is valid in arbitrary depth.

\begin{thm}\label{thm:weak-parity}
If $k\equiv d\pmod2$, then
\begin{align}\label{eq:weak-parity}
 \Fil_d\ZAfw{k}\subseteq\Fil_{d-1}\ZAfw{k}
 +\sum_{\substack{k_1+k_2=k\\d_1+d_2=d\\d_1,d_2\geq2}}
 \Fil_{d_1}\ZAfw{k_1}\,\Fil_{d_2}\ZAfw{k_2}.
\end{align}
\end{thm}

\begin{proof}
Let $w=z_{a_1}\cdots z_{a_d}$ have weight $k$. Apply the quotient map to \eqref{eq:antipode-identity}. Modulo $\Fil_{d-1}\ZAfw{k}$, contractions in $\Sigma$ disappear, and we obtain
\[
 0\equiv\sum_{i=0}^d(-1)^i
 \zaf{a_i,\ldots,a_1}\zaf{a_{i+1},\ldots,a_d}.
\]
The two end terms are
\[
 \zaf{a_1,\ldots,a_d}+(-1)^d\zaf{a_d,\ldots,a_1}
 =\bigl(1+(-1)^{k+d}\bigr)\zaf{a_1,\ldots,a_d}
\]
by reversal. If $k\equiv d\pmod2$, their sum is twice the original value. The terms with a factor of depth one vanish by \cref{prop:low-depth}, leaving exactly the products in \eqref{eq:weak-parity}.
\end{proof}

In depths three and four the preceding proof gives explicit formulas.

\begin{cor}\label{cor:weak-low-depth}
If $a+b+c$ is odd, then
\begin{align}\label{eq:weak-depth-three}
 \zaf{a,b,c}
 =-\frac12\bigl(\zaf{a+b,c}+\zaf{a,b+c}\bigr).
\end{align}
If $a+b+c+d$ is even, then
\begin{align}
 \zaf{a,b,c,d}
 &=-\frac12\bigl(\zaf{a+b,c,d}+\zaf{a,b+c,d}
 +\zaf{a,b,c+d}+(-1)^{a+b}\zaf{a,b}\zaf{c,d}\bigr).
 \label{eq:weak-depth-four}
\end{align}
\end{cor}

\begin{proof}
For $w=z_az_bz_c$, the terms at the first and second cut in \eqref{eq:antipode-identity} contain a depth-one factor and vanish. Expanding the antipode of $w$ therefore gives
\[
 0=\zaf{a,b,c}-\zaf{c,b,a}-\zaf{b+c,a}
   -\zaf{c,a+b}-\zaf{a+b+c}.
\]
The last term vanishes, and reversal in odd total weight gives \eqref{eq:weak-depth-three}.

For $w=z_az_bz_cz_d$, again the terms at the first and third cuts vanish. At the middle cut one has
\[
 S_*(z_az_b)*z_cz_d=(z_bz_a+z_{a+b})*z_cz_d,
\]
whose second summand vanishes because $\zaf{a+b}=0$. In the antipode of the full word, all terms with at least two contractions have depth at most two and even weight, and hence vanish by \cref{prop:low-depth}. The remaining identity is
\[\begin{aligned}
0={}&\zaf{a,b,c,d}+\zaf{b,a}\zaf{c,d}+\zaf{d,c,b,a}\\
 &+\zaf{c+d,b,a}+\zaf{d,b+c,a}+\zaf{d,c,a+b}.
\end{aligned}\]
Reversal in even total weight turns the last four terms into the original depth-four value and the three depth-three terms displayed in \eqref{eq:weak-depth-four}. Together with $\zaf{b,a}=(-1)^{a+b}\zaf{a,b}$, this proves the formula.
\end{proof}

The product in \eqref{eq:weak-depth-four} is the only obstruction to strong parity in depth four.

\begin{conj}\label{conj:strong-parity}
If $k\equiv d\pmod2$, then
\[
 \Fil_d\ZAfw{k}\subseteq\Fil_{d-1}\ZAfw{k}.
\]
\end{conj}

For the symmetric realization, the image of this conjecture is known in every depth by~\cite[Corollary to Theorem~4]{KZ}. Note that the result does not lift back through the surjection $\pi_{\mathcal S}$, and the corresponding finite statement in arbitrary depth remains open. A proof in $\ZAf$, on the other hand, gives both realizations at once.

\begin{thm}
\label{thm:strong-parity-four}
\Cref{conj:strong-parity} holds for $d\leq4$.
\end{thm}

\begin{proof}
For $d=1$, all depth-one values vanish by \cref{prop:low-depth}. For $d=2$, the relevant weight is even, and every depth-two value vanishes by the same proposition. For $d=3$, the relevant weight is odd, and \eqref{eq:weak-depth-three} expresses every depth-three value in depth at most two. Thus only depth four requires an argument.

Fix an even weight $k$. Write $r=a+b$ and $s=c+d$, so that $r+s=k$. Formula \eqref{eq:weak-depth-four} and the depth-two evaluation \eqref{eq:formal-double-evaluation} give
\begin{align}\label{eq:parity-dictionary}
 \zaf{a,b,c,d}\equiv
 \begin{cases}
  \dfrac12(-1)^{a+c}\binom{r}{a}\binom{s}{c}\zzf r\zzf s,
       &r\text{ odd},\\[3pt]
  0,&r\text{ even},
 \end{cases}
 \pmod{\Fil_3\ZAfw{k}}.
\end{align}
Thus it is enough to prove that every product $\zzf r\zzf s$, where $r,s\geq3$ are odd and $r+s=k$, belongs to $\Fil_3\ZAfw{k}$.

For odd $r\geq3$, set
\begin{align}\label{eq:odd-f-polynomial}
 f_r(X,Y)=\frac{(Y-X)^r-Y^r+X^r}{XY}
 =\sum_{j=1}^{r-1}(-1)^j\binom rj X^{j-1}Y^{r-1-j}.
\end{align}
Let
\[
 \mathfrak F_k(X,Y,Z,W)=
 \sum_{a+b+c+d=k}\zaf{a,b,c,d}
 X^{a-1}Y^{b-1}Z^{c-1}W^{d-1}.
\]
Then \eqref{eq:parity-dictionary} is equivalent to
\begin{align}\label{eq:F-Q-congruence}
 \mathfrak F_k(X,Y,Z,W)&\equiv\frac12\mathfrak Q_k(X,Y,Z,W)
       \pmod{\Fil_3\ZAfw{k}},\\
 \mathfrak Q_k(X,Y,Z,W)&=
 \sum_{\substack{r+s=k\\r,s\geq3\text{ odd}}}
 \zzf r\zzf s f_r(X,Y)f_s(Z,W),\notag
\end{align}
where the sum is over the ordered pairs $(r,s)$.

We use only the linear shuffle relation with a word of depth one and a word of depth three. Expanding \eqref{eq:formal-linear-shuffle} for $w=z_a$ and $v=z_bz_cz_d$, multiplying by $(-X)^{a-1}Y^{b-1}Z^{c-1}W^{d-1}$ and summing, gives the exact identity
\begin{align}
0={}&\mathfrak F_k(X,Y,Z,W)+\mathfrak F_k(Y-X,Y,Z,W)\notag\\
 &+\mathfrak F_k(Y-X,Z-X,Z,W)
 +\mathfrak F_k(Y-X,Z-X,W-X,W)\notag\\
 &+\mathfrak F_k(Y-X,Z-X,W-X,-X).
\label{eq:five-term-linear-shuffle}
\end{align}
Note that this gives a direct proof of the generating-series identity, without using a separate shuffle formula.

Insert \eqref{eq:F-Q-congruence} into \eqref{eq:five-term-linear-shuffle} and set $X=W=0$. Since
\[
 f_r(0,Y)=-rY^{r-2},\qquad f_r(Y,Y)=0,
 \qquad f_s(Z,0)=sZ^{s-2},
\]
only the first and third substitutions survive. We obtain
\begin{align}
0\equiv{}&\sum_{\substack{r+s=k\\r,s\geq3\text{ odd}}}
 \zzf r\zzf s\,sZ^{s-2}\bigl(f_r(Y,Z)-rY^{r-2}\bigr)\notag\\
={}&\sum_{\substack{r+s=k\\r,s\geq3\text{ odd}}}
 \zzf r\zzf s\,sZ^{s-2}
 \sum_{j=1}^{r-2}(-1)^j\binom rjY^{j-1}Z^{r-1-j}
 \pmod{\Fil_3\ZAfw{k}}.
\label{eq:triangular-generating-series}
\end{align}
The coefficient of $Y^{\alpha-1}Z^{k-\alpha-3}$ yields
\begin{align}\label{eq:triangular-product-system}
 \sum_{\substack{r+s=k,\ r,s\geq3\text{ odd}\\\alpha\leq r-2}}
 s\binom r\alpha\zzf r\zzf s\equiv0
 \pmod{\Fil_3\ZAfw{k}}.
\end{align}
Choose successively $\alpha=r_0-2$ for $r_0=k-3,k-5,\ldots$ down to the smallest odd $r_0\geq k/2$. In \eqref{eq:triangular-product-system}, pairs with $\max(r,s)<r_0$ do not occur, while pairs with $\max(r,s)>r_0$ have already been treated. The remaining coefficient of $\zzf{r_0}\zzf{k-r_0}$ is nonzero. Downward induction therefore gives $\zzf r\zzf s\in\Fil_3\ZAfw{k}$ for all pairs. From \eqref{eq:parity-dictionary} we get the result.
\end{proof}

\begin{rem}
We note that the original proof of \cref{thm:strong-parity-four} in \cite{R} differs from the one presented here. There the theorem was proved by an exhaustive computer-assisted calculation using group actions on generating series. As a result, one also obtains an explicit formula for the depth-four finite multiple zeta values in terms of depth three. We refer to \cite{R} for this formula.
\end{rem}

\begin{cor}
\label{cor:explicit-true-single-product}
Let $r,s\geq3$ be odd and put $k=r+s$. Then
\begin{align}\label{eq:four-one-product}
 \zaf{r-1,1,s-1,1}
 =\frac12\zaf{r-1,1}\zaf{s-1,1}
  +\frac12\bigl(\zaf{1,r-1,s}-\zaf{1,s-1,r}\bigr).
\end{align}
Consequently,
\begin{align}\label{eq:four-one-symmetry}
 \zaf{r-1,1,s-1,1}+\zaf{s-1,1,r-1,1}
 &=\zaf{r-1,1}\zaf{s-1,1},\\
 \zaf{r-1,1,s-1,1}-\zaf{s-1,1,r-1,1}
 &=\zaf{1,r-1,s}-\zaf{1,s-1,r}.\notag
\end{align}
\end{cor}

\begin{proof}
Apply \eqref{eq:weak-depth-four} to $\zaf{r-1,1,s-1,1}$. Since $r$ is odd, it gives
\[
 \zaf{r-1,1,s-1,1}
 =\frac12\zaf{r-1,1}\zaf{s-1,1}
 -\frac12\bigl(
 \zaf{r,s-1,1}+\zaf{r-1,s,1}+\zaf{r-1,1,s}\bigr).
\]
Reversal in the even weight $k$ identifies the first triple in parentheses with $\zaf{1,s-1,r}$. On the other hand, the stuffle expansion of $0=\zaf1\zaf{r-1,s}$, whose two depth-two contractions vanish by \cref{prop:low-depth}, gives
\[
 \zaf{r-1,s,1}+\zaf{r-1,1,s}=-\zaf{1,r-1,s}.
\]
This proves \eqref{eq:four-one-product}. Adding and subtracting it from the formula with $r$ and $s$ interchanged proves \eqref{eq:four-one-symmetry}.
\end{proof}

\section{The formal double zeta space}\label{sec:double-zeta}

We recall the formal double zeta space of Gangl, Kaneko and Zagier~\cite{GKZ}.

\begin{dfn}\label{def:double-zeta}
For $k\geq3$, the \emph{formal double zeta space} $\Dk$ is the $\QQ$-vector space generated by symbols
\[
 Z_{r,s},\quad P_{r,s}\quad(r,s\geq1,\ r+s=k),
 \qquad Z_k,
\]
subject to
\begin{align}\label{eq:formal-double-shuffle}
 P_{r,s}
 &=Z_{r,s}+Z_{s,r}+Z_k,\\
 &=\sum_{a+b=k}\left[\binom{a-1}{r-1}+\binom{a-1}{s-1}\right]Z_{a,b}.\notag
\end{align}
A realization of $\Dk$ in a $\QQ$-vector space $V$ is a linear map $\Dk\to V$, i.e. an assignment satisfying \eqref{eq:formal-double-shuffle}.
\end{dfn}

For even $k$, the quotient relevant to finite multiple zeta values is
\begin{align}\label{eq:reduced-double-zeta}
 \Dtk=\Dk/\mathcal P_k,
 \qquad
 \mathcal P_k=\QQ Z_k+\QQ Z_{1,k-1}
 +\Span_{\QQ}\{P_{r,s}:r,s\text{ even}\}.
\end{align}
We continue to write $Z_{r,s}$ and $P_{r,s}$ for their classes in $\Dtk$. Since $Z_k=0$ there, the first relation in \eqref{eq:formal-double-shuffle} becomes
\begin{align}\label{eq:product-in-Dtilde}
 P_{r,s}=Z_{r,s}+Z_{s,r}\qquad(r+s=k).
\end{align}
Here $Z_k$ and the product symbols $P_{r,s}$ with $r,s$ even model terms divisible by $\zeta(2)$, while $Z_{1,k-1}$ corresponds to the regularized boundary case.

\begin{prop}\label{prop:dimension-Dtilde}
For every even $k\geq4$,
\begin{align}\label{eq:dimension-Dtilde}
 \dim_{\QQ}\Dtk=\frac{k}{2}-2-\dim_{\QQ}\mathcal S_k.
\end{align}
\end{prop}

\begin{proof}
In even weight the $k/2$ odd--odd symbols form a basis of $\Dk$. The period-polynomial theorem of~\cite{GKZ} shows that
\[
 \dim_{\QQ}\left(\QQ Z_k+
 \Span_{\QQ}\{P_{r,k-r}:r\text{ even}\}\right)
 =\dim_{\QQ}\mathcal S_k+1.
\]
Here
\[
 Z_k=\frac{2}{k+1}
 \sum_{\substack{2\leq r\leq k-2\\r\ {\rm even}}}P_{r,k-r},
\]
so $Z_k$ already belongs to the span of the even product symbols. Moreover, $Z_{1,k-1}$ is independent of this subspace: in the $\kappa$-dependent realization of~\cite[(24)]{GKZ}, its image depends on $\kappa$, whereas the images of $Z_k$ and the even product symbols do not. Hence
\[
 \dim_{\QQ}\Dtk=\frac{k}{2}
 -\bigl(\dim_{\QQ}\mathcal S_k+1\bigr)-1,
\]
which is \eqref{eq:dimension-Dtilde}.
\end{proof}

The proof of \cref{prop:dimension-Dtilde} rests on the correspondence, established in~\cite{GKZ}, between the even period polynomials \eqref{eq:even-period-polynomials} and relations in $\Dk$. In $\Dtk$ this correspondence takes the following explicit form.

\begin{prop}\label{prop:period-polynomial-relations}
Let $k\geq4$ be even. For $p\in W_k^{\even}$ define $q_{r,s}\in\QQ$ for $r+s=k$ by
\[
 p(X+Y,Y)=\sum_{r+s=k}\binom{k-2}{r-1}q_{r,s}X^{r-1}Y^{s-1}.
\]
Then
\begin{align}\label{eq:period-polynomial-relation}
 \sum_{\substack{r+s=k\\r,s\geq3\ \odd}}
 \bigl(q_{r,s}-q_{k-1,1}\bigr)Z_{r,s}=0
 \qquad\text{in }\Dtk,
\end{align}
and every linear relation among the classes $Z_{r,s}$ with $r,s\geq3$ odd in $\Dtk$ is of this form. The kernel of the assignment $p\mapsto\eqref{eq:period-polynomial-relation}$ is the line $\QQ(X^{k-2}-Y^{k-2})$, so that these relations form a space of dimension $\dim_{\QQ}\mathcal S_k$.
\end{prop}

\begin{proof}
Write $p(X,Y)=\sum_{r+s=k}\binom{k-2}{r-1}p_{r,s}X^{r-1}Y^{s-1}$. By \cite[Theorem~3]{GKZ}, one has $q_{r,s}-q_{s,r}=p_{r,s}$ and
\begin{align}\label{eq:GKZ-theorem-three}
 \sum_{\substack{r+s=k\\r,s\ \even}}q_{r,s}Z_{r,s}
 \equiv 3\sum_{\substack{r+s=k\\r,s\ \odd}}q_{r,s}Z_{r,s}
 \pmod{\QQ Z_k}
\end{align}
in $\Dk$. Conversely, an element of $\Dk$, written in the basis $\{Z_{r,s}:r,s\ \odd\}$, lies in $\QQ Z_k+\Span_{\QQ}\{P_{r,s}:r,s\ \even\}$ if and only if its coefficient vector is $(q_{r,s})_{r,s\ \odd}$ for a (unique) $p\in W_k^{\even}$.

Since $p$ contains only even powers, $p_{r,s}=0$ for even $r$ and $s$, so that $q_{r,s}=q_{s,r}$ there, and
\[
 2\sum_{\substack{r+s=k\\r,s\ \even}}q_{r,s}Z_{r,s}
 =\sum_{\substack{r+s=k\\r,s\ \even}}q_{r,s}\bigl(P_{r,s}-Z_k\bigr).
\]
In $\Dtk$ the even product symbols and $Z_k$ vanish, and \eqref{eq:GKZ-theorem-three} becomes
\begin{align}\label{eq:period-open-relation}
 \sum_{\substack{r+s=k\\r,s\ \odd}}q_{r,s}Z_{r,s}=0
 \qquad\text{in }\Dtk.
\end{align}
The summand with $r=1$ vanishes. For the summand with $r=k-1$ we use the sum formula in $\Dtk$: the relation \eqref{eq:formal-double-shuffle} for $\{r,s\}=\{1,k-1\}$ reads $P_{1,k-1}=Z_{k-1,1}+\sum_{a=1}^{k-1}Z_{a,k-a}$, while $P_{1,k-1}=Z_{1,k-1}+Z_{k-1,1}+Z_k=Z_{k-1,1}$ in $\Dtk$. Hence $\sum_{a=1}^{k-1}Z_{a,k-a}=0$. Removing the even part, which vanishes as before, gives
\[
 Z_{k-1,1}=-\sum_{\substack{3\leq a\leq k-3\\a\ \odd}}Z_{a,k-a}.
\]
Substituting this into \eqref{eq:period-open-relation} yields \eqref{eq:period-polynomial-relation}.

Conversely, suppose that $\sum c_{r,s}Z_{r,s}=0$ in $\Dtk$, the sum being over odd $r,s\geq3$. By the definition \eqref{eq:reduced-double-zeta} of $\Dtk$ there is a $\mu\in\QQ$ with
\[
 \sum_{\substack{r+s=k\\r,s\geq3\ \odd}}c_{r,s}Z_{r,s}
 -\mu Z_{1,k-1}
 \in\QQ Z_k+\Span_{\QQ}\{P_{r,s}:r,s\ \even\}
\]
in $\Dk$. By the converse part of \cite[Theorem~3]{GKZ} recalled above, comparing coefficients in the basis $\{Z_{r,s}:r,s\ \odd\}$ gives a $p\in W_k^{\even}$ with $c_{r,s}=q_{r,s}$ for $r,s\geq3$, with $q_{k-1,1}=0$ and $q_{1,k-1}=-\mu$. Hence $c_{r,s}=q_{r,s}-q_{k-1,1}$, as claimed.

Finally, $p=X^{k-2}-Y^{k-2}$ satisfies $p(X+Y,Y)=(X+Y)^{k-2}-Y^{k-2}$, so that $q_{r,s}=1$ for $r\geq2$ and $q_{1,k-1}=0$. Thus $p$ lies in the kernel. The classes $Z_{r,s}$ with $r,s\geq3$ odd span $\Dtk$, since the odd--odd symbols form a basis of $\Dk$, and $Z_{1,k-1}$ and $Z_{k-1,1}$ have just been expressed in terms of the remaining ones. The space of relations among these spanning classes therefore has dimension $(\frac k2-2)-\dim_{\QQ}\Dtk=\dim_{\QQ}\mathcal S_k =\dim_{\QQ}W_k^{\even}-1$. By rank--nullity, the kernel is exactly the line $\QQ(X^{k-2}-Y^{k-2})$.
\end{proof}

By strong parity, $\Fil_4\ZAfw{k}=\Fil_3\ZAfw{k}$ in even weight. We use depth-four representatives because the map below is stated in terms of $\zaf{r-1,1,s-1,1}$.

\subsection{A generating family in \texorpdfstring{$\Dtk$}{D-tilde}}

Fix an even integer $k\geq4$ and recall the elements $H_{r,s}$ from \eqref{eq:H-source-element}. By \eqref{eq:product-in-Dtilde}, they satisfy
\begin{align}\label{eq:H-source-symmetry}
 H_{r,s}+H_{s,r}=rsP_{r,s}.
\end{align}

We first record the polynomial lemma used below. Let $V$ be a $\QQ$-vector space and let $Z\in V[X,Y]$ be homogeneous of degree $k-2$. Define
\begin{align}\label{eq:H-source-operator}
 \mathcal L(Z)=(X^2+Y^2)Z+XY(X+Y)(\partial_X+\partial_Y)Z.
\end{align}

\begin{lem}\label{lem:dihedral-odd-coefficients}
Suppose that $k\geq4$ is even and
\begin{align}\label{eq:dihedral-anti-invariance}
 Z(Y,X)=-Z(X,Y),\qquad Z(X,X-Y)=-Z(X,Y).
\end{align}
If
\[
 [X^rY^{k-r}]\mathcal L(Z)=0
 \qquad(3\leq r\leq k-3,\ r\text{ odd}),
\]
then $Z=0$.
\end{lem}

\begin{proof}
The two transformations in \eqref{eq:dihedral-anti-invariance} generate a dihedral reflection group of order $12$, of type $I_2(6)$, containing six reflections. Its ring of invariants is $\QQ[A_2,A_6]$, and the anti-invariant polynomials form the free module over it generated by the product $J$ of the six linear forms vanishing on the reflection lines~\cite[Chapter~3]{Hu}, where
\[
 J=XY(X-Y)(X+Y)(2X-Y)(X-2Y),\quad
 A_2=X^2-XY+Y^2,\quad A_6=X^2Y^2(X-Y)^2.
\]
A polynomial which is anti-invariant under this group is therefore of the form
\begin{align}\label{eq:dihedral-factorization}
 Z=J\sum_{\substack{b\geq0\\6b\leq k-8}}
 v_bA_2^{(k-8-6b)/2}A_6^b,
\end{align}
with $v_b\in V$. For $k=4,6$ the sum is empty, and there is nothing to prove.

Set $m=k-2$, $z(t)=Z(t,1)$ and $w(t)=\mathcal L(Z)(t,1)$. Homogeneity gives
\begin{align}\label{eq:dihedral-one-variable}
 w(t)=\bigl((m+1)t^2+mt+1\bigr)z(t)+(t-t^3)z'(t).
\end{align}
If $z(t)=\sum c_it^i$ and $w(t)=\sum d_jt^j$, then
\begin{align}\label{eq:dihedral-recurrence}
 d_j=(j+1)c_j+mc_{j-1}+(m-j+3)c_{j-2}.
\end{align}
The polynomial $\mathcal L(Z)$ is antisymmetric. Hence the assumption on its odd coefficients gives
\begin{align}\label{eq:dihedral-boundary}
 w_{\mathrm{odd}}(t)=a(t-t^{m+1}),\qquad a=2z'(0).
\end{align}

Writing \eqref{eq:dihedral-factorization} at $Y=1$ gives
\[
 z(t)=J(t)F\bigl(t^2-t+1,t^2(t-1)^2\bigr),
 \qquad J(t)=J(t,1)=t(t-1)(t+1)(2t-1)(t-2).
\]
Since $J'(0)=J'(1)=-2$, we have $z'(1)=z'(0)$. Equation \eqref{eq:dihedral-one-variable} now gives $w'(1)=ma$ and $w'(-1)=0$. Thus $w_{\mathrm{odd}}'(1)=ma/2$, whereas \eqref{eq:dihedral-boundary} gives $w_{\mathrm{odd}}'(1)=-ma$. It follows that $a=0$.

Order the terms in \eqref{eq:dihedral-factorization} by $b$. The $b$-th term begins with $-2v_bt^{2b+1}$. Once the preceding terms have vanished, \eqref{eq:dihedral-recurrence} with $j=2b+1$ gives $-4(b+1)v_b=0$. Induction gives $v_b=0$ for every $b$.
\end{proof}

\begin{prop}\label{prop:H-spans-Dtilde}
For every even $k\geq4$, the elements
\[
 H_{r,s}\qquad(r,s\geq3\text{ odd},\ r+s=k)
\]
span $\Dtk$.
\end{prop}

\begin{proof}
Work in the quotient of $\Dtk$ by the span of the $H_{r,s}$. By \eqref{eq:H-source-symmetry}, $P_{r,k-r}=0$ for odd $r$ with $3\leq r\leq k-3$. The same equality for even $r$ follows from the definition of $\Dtk$. The double shuffle relation for the pair $\{1,k-1\}$, together with $Z_{1,k-1}=0$, gives
\[
 \sum_{r=1}^{k-1}Z_{r,k-r}=0.
\]
Pairing the terms indexed by $r$ and $k-r$ shows that $P_{1,k-1}=0$. Thus
\[
 P_{r,k-r}=0\qquad(1\leq r\leq k-1).
\]

Form
\[
 Z(X,Y)=\sum_{r=1}^{k-1}Z_{r,k-r}X^{r-1}Y^{k-r-1}.
\]
These equalities and \eqref{eq:formal-double-shuffle} give
\[
 Z(Y,X)=-Z(X,Y),\qquad Z(X,X-Y)=-Z(X,Y).
\]
Using $P_{a,k-a}=0$, and hence $Z_{a,k-a}=-Z_{k-a,a}$, in \eqref{eq:H-source-element} we obtain
\[
 2H_{r,s}=-(s+1)Z_{r-1,s+1}-(k-2)Z_{r,s}
 -(r+1)Z_{r+1,s-1}.
\]
The negative of the right-hand side is $[X^rY^s]\mathcal L(Z)$. All the required coefficients of $\mathcal L(Z)$ therefore vanish, and \cref{lem:dihedral-odd-coefficients} gives $Z=0$.
\end{proof}

Since there are $k/2-2$ elements $H_{r,s}$, \cref{prop:dimension-Dtilde,prop:H-spans-Dtilde} also give
\begin{align}\label{eq:H-relation-dimension}
 \dim_{\QQ}\left\{(c_{r,s}):
 \sum_{r+s=k}c_{r,s}H_{r,s}=0\right\}
 =\dim_{\QQ}\mathcal S_k,
\end{align}
where the sum is over odd $r,s\geq3$.

\subsection{The maps \texorpdfstring{$\beta_k$ and $\beta_k^{\mathrm f}$}{beta-k and beta-k-f}}

The elements $H_{r,s}$ span $\Dtk$ by \cref{prop:H-spans-Dtilde}. Hence, for odd $r,s\geq3$ with $r+s=k$, the rules
\begin{align}\label{eq:direct-beta-rules}
 \beta_k^{\mathrm f}(H_{r,s})
 &=\zaf{r-1,1,s-1,1},\notag\\
 \beta_k(H_{r,s})
 &=\za{r-1,1,s-1,1}
\end{align}
determine at most one linear map in each case, and such a map exists if and only if the right-hand sides satisfy every linear relation among the $H_{r,s}$. If $\mathcal S_k=0$, the elements $H_{r,s}$ form a basis by \cref{eq:dimension-Dtilde,prop:H-spans-Dtilde}, and both rules define a map without any further assumption. In the cusp weights this is a nontrivial condition.

The construction of these maps uses the following two identities.

\begin{thm}\label{thm:depth-three-reduction}
For every even $k\geq4$ and all $a,b,c\geq1$ with $a+b+c=k$, one has in $\ZAfw{k}$
\begin{align}\label{eq:formal-Hoffman-depth-three}
 \zaf{a,b,c}
={}&(-1)^{b-1}
 \sum_{\substack{n+m=k-1\\n,m\geq1}}
 \binom na\binom mc\,\zaf{n,1,m}\notag\\
&+(-1)^{a+c}\sum_{i=1}^{b-1}
 \binom{a+i}{a}\binom{b+c-i}{c}
 \zzf{a+i}\zzf{b+c-i}.
\end{align}
\end{thm}

\begin{proof}
We first record a consequence of the defining relations. For a nonempty word $w=z_{a_1}\cdots z_{a_d}$, put
\begin{align*}
 D_1(w)={}&z_1w
 +\sum_{j=1}^d\sum_{i=1}^{a_j-1}
 z_{a_1}\cdots z_{a_{j-1}}z_{i+1}z_{a_j-i}z_{a_{j+1}}\cdots z_{a_d}\\
 &-\sum_{j=1}^d z_{a_1}\cdots z_{a_j+1}\cdots z_{a_d}.
\end{align*}
Comparing the shuffle product and the stuffle product with $z_1$, we get
\begin{align}\label{eq:first-derivation-word}
 D_1(w)=(z_1\shuffle w+z_1w)-z_1*w.
\end{align}
The term in parentheses is the generator \eqref{eq:formal-linear-shuffle} of $\mathfrak L$ for the words $z_1$ and $w$, and $z_1*w$ vanishes in $\ZAf$ since $\zaf1=0$. Therefore $D_1(w)$ vanishes in $\ZAf$.

Fix an even weight $k$ and write $A(a,b,c)=\zaf{a,b,c}$ for indices of weight $k$. Since $\zaf{a,b}=0$ for even $a+b$, the formula \eqref{eq:weak-depth-four} can be written as
\begin{align}\label{eq:triple-proof-antipode}
 2\zaf{a,b,c,d}={}&\zaf{a,b}\zaf{c,d}
 -A(a+b,c,d)\notag\\
 &-A(a,b+c,d)-A(a,b,c+d).
\end{align}

For $t\geq2$ and $s\geq1$ with $t+s=k$, set $F(t;s)=\sum_{u+v=t}A(u,v,s)$, where $u,v\geq1$. Applying \eqref{eq:first-derivation-word} to $w=z_pz_q$ with $p+q=k-1$ and using reversal, we get
\begin{align}\label{eq:triple-proof-split-sum}
 F(p+1;q)+F(q+1;p)=A(p,1,q).
\end{align}
Here the terms of $D_1(w)$ in which one entry is increased have depth two and vanish. Next we apply the same identity to $w=z_az_bz_c$ with $a+b+c=k-1$, substitute \eqref{eq:triple-proof-antipode} in all depth-four terms and multiply by $2$. The product terms coming from splitting the first entry combine into $(\sum_{u+v=a+1}\zaf{u,v})\zaf{b,c}=0$, and by the sum formula those coming from splitting the last entry give $-\zaf{a,b}\zaf{1,c}$. Using reversal to collect the remaining terms, we obtain
\begin{align*}
0={}&\sum_{i=1}^{b-1}\zaf{a,i+1}\zaf{b-i,c}-\zaf{a,b}\zaf{1,c}\\
 &-(a+1)A(a+1,b,c)-bA(a,b+1,c)-(c+1)A(a,b,c+1)\\
 &-\bigl(F(a+b+1;c)+F(c+1;a+b)-A(a+b,1,c)\bigr)\\
 &-\bigl(F(a+1;b+c)+F(b+c+1;a)-A(a,1,b+c)\bigr).
\end{align*}
The last two lines vanish by \eqref{eq:triple-proof-split-sum}, so
\begin{align}\label{eq:inhomogeneous-triple-recurrence}
 &(a+1)A(a+1,b,c)+bA(a,b+1,c)+(c+1)A(a,b,c+1)\notag\\
 &\qquad=\sum_{i=1}^{b-1}\zaf{a,i+1}\zaf{b-i,c}-\zaf{a,b}\zaf{1,c}.
\end{align}

For $a+b+c=k$, define the product correction
\[
 P(a,b,c)=(-1)^{a+c}\sum_{i=1}^{b-1}
 \binom{a+i}{a}\binom{b+c-i}{c}\zzf{a+i}\zzf{b+c-i}.
\]
It satisfies the same recurrence \eqref{eq:inhomogeneous-triple-recurrence} as $A$. To see this, suppose first that $b\geq2$ and $a+b+c=k-1$. The last summand on the right of \eqref{eq:inhomogeneous-triple-recurrence} cancels $\zaf{a,b}\zaf{1,c}$. If we substitute $P$ on the left-hand side, the two endpoint coefficients vanish and the interior coefficients reduce to
\[
 (-1)^{a+c+1}\sum_{i=1}^{b-2}
 \binom{a+i+1}{a}\binom{b+c-i}{c}
 \zzf{a+i+1}\zzf{b+c-i}.
\]
Here we used $(a+1)\binom{u}{a+1}=(u-a)\binom ua$, its counterpart for $c$, and $-(u-a)+b-(v-c)=-1$ for $u+v=k$. By \eqref{eq:formal-double-evaluation}, this sum equals the remaining products on the right-hand side, since $\zzf{n}=0$ for even $n$. If $b=1$, the required identity is $P(a,2,c)=-\zaf{a,1}\zaf{1,c}$, which follows from the same evaluation.

The difference $A-P$ therefore satisfies the homogeneous recurrence. For $b>1$, this reads
\begin{align*}
 (A-P)(a,b,c)={}&-\frac{a+1}{b-1}(A-P)(a+1,b-1,c)\\
 &-\frac{c+1}{b-1}(A-P)(a,b-1,c+1).
\end{align*}
Since $P(n,1,m)=0$, iteration gives
\[
 (A-P)(a,b,c)=(-1)^{b-1}\sum_{j=0}^{b-1}
 \binom{a+j}{a}\binom{c+b-1-j}{c}A(a+j,1,c+b-1-j).
\]
The coefficient is obtained by choosing the $j$ steps which increase the first entry and dividing the product of the increment factors by $(b-1)!$. Setting $n=a+j$ and $m=c+b-1-j$ proves \eqref{eq:formal-Hoffman-depth-three}.
\end{proof}

The second identity used below, the \emph{harmonic identity}, follows from \cref{thm:depth-three-reduction}.

\begin{prop}\label{prop:harmonic-from-hoffman}
Let $k\geq4$ be even. For $2\leq r,s\leq k-2$ with $r+s=k$, one has in $\ZAfw{k}$
\begin{align}\label{eq:formal-harmonic}
0=
 \sum_{\substack{i+j=r,\ a+b=s\\i,j,a,b\geq1\\i+a\ \mathrm{odd}}}
 \zaf{i,a}\zaf{j,b}
 +2(-1)^r\sum_{\substack{u+v=k-1\\u,v\geq1}}
 \left(\binom vr+\binom vs\right)\zaf{u,1,v}.
\end{align}
The same identity holds for finite multiple zeta values.
\end{prop}

\begin{proof}
All identities are read in weight $k$. Write
\[
 A(X,Y,Z)=\sum_{a+b+c=k}\zaf{a,b,c}X^{a-1}Y^{b-1}Z^{c-1},
 \qquad
 T(U,V)=\sum_{\substack{u+v=k-1\\u,v\geq1}}\zaf{u,1,v}U^{u-1}V^{v-1},
\]
and set $N_n(X,Y)=(X-Y)^n-(-Y)^n$. Multiplying \eqref{eq:formal-Hoffman-depth-three} by $X^{a-1}Y^{b-1}Z^{c-1}$ and summing over $a,b,c\geq1$ with $a+b+c=k$ gives
\begin{align}\label{eq:hoffman-summed}
 A(X,Y,Z)={}&\sum_{\substack{n+m=k-1\\n,m\geq1}}\zaf{n,1,m}\,
 \frac{N_n(X,Y)}{X}\,\frac{N_m(Z,Y)}{Z}\notag\\
 &+Y\sum_{\substack{u+v=k\\u,v\ \odd}}
 \zzf u\zzf v\,f_u(X,Y)f_v(Z,Y),
\end{align}
where $f_u$ is the polynomial \eqref{eq:odd-f-polynomial}. Indeed, the binomial sums over $a$ and $c$ produce the first term by
\[
 \sum_{a=1}^{n}\binom naX^{a-1}(-Y)^{n-a}=\frac{N_n(X,Y)}{X},
\]
and in the second term we used, for odd $u$,
\[
 \sum_{a=1}^{u-1}\binom ua(-X)^aY^{u-a}=XY\,f_u(X,Y),
\]
together with $\zzf u=0$ for even $u$ and $\zzf1=0$.

Expanding the products $0=\zaf a\zaf{b,c}$ by the stuffle product and using that all depth-two values of even weight vanish gives $A(X,Y,Z)+A(Y,X,Z)+A(Y,Z,X)=0$, and in particular
\begin{align}\label{eq:cyclic-specialized}
 A(X,Y,X)+2A(Y,X,X)=0.
\end{align}
We substitute \eqref{eq:hoffman-summed} into \eqref{eq:cyclic-specialized} and multiply by $X^2Y$. By \cref{lem:restricted-sum} and reversal one has $\sum_{u+v=k-1}\zaf{u,1,v}=0$, $T(U,V)=T(V,U)$ and $T(-U,-V)=-T(U,V)$, and moreover $f_v(X,X)=0$ and $N_m(X,X)=(-1)^{m+1}X^m$. In particular, the product term of $A(Y,X,X)$ vanishes, so the product terms contribute exactly $X^2Y^2\sum\zzf u\zzf v f_u(X,Y)f_v(X,Y)$. For the terms involving $T$, expanding $N_nN_m$ and using the three rules gives
\begin{align*}
 Y\sum_{n+m=k-1}\zaf{n,1,m}\,N_n(X,Y)N_m(X,Y)
 &=2Y^2(Y-X)\,T(Y,Y-X),\\
 2\sum_{n+m=k-1}\zaf{n,1,m}\,(-1)^{m+1}N_n(Y,X)X^{m+1}
 &=2X^2(X-Y)\,T(X,X-Y),
\end{align*}
the two left-hand sides being the parts involving $T$ in $X^2Y\,A(X,Y,X)$ and in $2X^2Y\,A(Y,X,X)$. Altogether this yields
\begin{align}\label{eq:harmonic-boxed}
 X^2Y^2\sum_{\substack{u+v=k\\u,v\ \odd}}
 \zzf u\zzf v\,f_u(X,Y)f_v(X,Y)
 &+2Y^2(Y-X)\,T(Y,Y-X)\notag\\
 &+2X^2(X-Y)\,T(X,X-Y)=0.
\end{align}
For $2\leq r\leq k-2$, the coefficient of $X^rY^{k-r}$ in \eqref{eq:harmonic-boxed} is exactly \eqref{eq:formal-harmonic}, after evaluating the products $\zaf{i,a}\zaf{j,b}$ by \eqref{eq:formal-double-evaluation}. The coefficients of $X^rY^{k-r}$ with $r\in\{0,1,k-1,k\}$ vanish identically by \cref{lem:restricted-sum} and reversal. Every step used only the defining relations of $\ZAf$ together with identities proved in \cref{sec:formal-finite,sec:parity}, so the argument applies verbatim after the realization $\pi_{\mathcal A}$.
\end{proof}

\begin{cor}\label{cor:finite-harmonic}
The identities \eqref{eq:formal-Hoffman-depth-three} and \eqref{eq:formal-harmonic} hold for finite and symmetric multiple zeta values in every even weight.
\end{cor}

\begin{proof}
Apply the realization maps of \cref{prop:realizations} to \cref{thm:depth-three-reduction,prop:harmonic-from-hoffman}.
\end{proof}
Notice that for finite multiple zeta values, \eqref{eq:formal-Hoffman-depth-three} also follows from Hoffman's duality together with formulas of Kina (see \cite{H2} and \cite[Lemma~4.1, equation~(5.2) and Section~5]{Ki}). The proof above works directly with the defining relations of $\ZAf$.

We will use the following restricted sums. They follow directly from the formal sum formula, reversal, and the product with $\zaf1=0$.

\begin{lem}\label{lem:restricted-sum}
For every even $k\geq4$,
\begin{align}\label{eq:restricted-sum}
 \sum_{a+b=k-1}\zaf{1,a,b}
 =\sum_{a+b=k-1}\zaf{a,1,b}
 =\sum_{a+b=k-1}\zaf{a,b,1}=0,
\end{align}
where the summation variables are positive.
\end{lem}

\begin{proof}
Let $F$ be the sum of all depth-three values of weight $k$, and denote the three sums in \eqref{eq:restricted-sum} by $L,M,R$. The formal sum formula gives $F=0$. Summing the depth $1+2$ linear shuffle relation gives $2F+L+R=0$, while reversal gives $L=R$. Hence $L=R=0$. Finally, the sum of $0=\zaf1\zaf{a,b}$ over $a+b=k-1$ gives $L+M+R=0$. The contraction terms have depth two and even weight and therefore vanish. Thus $M=0$.
\end{proof}

The next lemma constructs the formal map in \eqref{eq:direct-beta-rules}.

\begin{lem}\label{lem:direct-beta-certificate}
For every even $k\geq4$, the first rule in \eqref{eq:direct-beta-rules} defines a linear map
\[
 \beta_k^{\mathrm f}:\Dtk\longrightarrow\Fil_4\ZAfw{k}.
\]
Its image is $\Fil_3\ZAfw{k}$.
\end{lem}

\begin{proof}
We verify compatibility by lifting the prescribed values to the standard source generators. Put
\[
 \mu_k=\frac12
 \sum_{\substack{1\leq n\leq k-1\\n\ \mathrm{odd}}}
 \zzf n\zzf{k-n}.
\]
For $1\leq r\leq k-1$, define
\begin{align}\label{eq:direct-beta-certificate}
 y_r=(-1)^r\sum_{n=1}^{k-2}
 \left(\binom nr-\frac12\binom{k-1}r\right)
 \left(\zaf{n,1,k-n-1}-\zzf n\zzf{k-n}\right)
 -\delta_{r,1}\mu_k.
\end{align}

Assign $Z_k\mapsto0$ and $Z_{r,k-r}\mapsto y_r$ for $1\leq r\leq k-1$, and abbreviate $D_n=\zaf{n,1,k-n-1}-\zzf n\zzf{k-n}$. We first evaluate the shuffle side $\sum_{a=1}^{k-1}\bigl[\binom{a-1}{r-1}+\binom{a-1}{s-1}\bigr]y_a$ of \eqref{eq:formal-double-shuffle}, where $s=k-r$. The inversion identity
\[
 \sum_{a\geq1}(-1)^a\binom{a-1}{r-1}\binom na=
 \begin{cases}
  (-1)^r,&n\geq r,\\
  0,&n<r,
 \end{cases}
\]
applied in \eqref{eq:direct-beta-certificate} gives
\[
 \sum_{a=1}^{k-1}\binom{a-1}{r-1}y_a
 =(-1)^r\sum_{n=r}^{k-2}D_n
 -\frac{(-1)^r}2\sum_{n=1}^{k-2}D_n
 -\delta_{r,1}\mu_k.
\]
By \cref{lem:restricted-sum} and reversal,
\begin{align}\label{eq:boundary-sums}
 \sum_{n=1}^{k-2}D_n=-2\mu_k,
 \qquad
 \sum_{n=r}^{k-2}D_n+\sum_{n=s}^{k-2}D_n=-2\mu_k-\zzf r\zzf s,
\end{align}
so the shuffle side equals $-(-1)^r\zzf r\zzf s-(\delta_{r,1}+\delta_{s,1})\mu_k$, that is, $\zzf r\zzf{k-r}$ for $2\leq r\leq k-2$ (both sides vanish for even $r$) and $-\mu_k$ for $\{r,k-r\}=\{1,k-1\}$. On the other hand, a Vandermonde convolution $\sum_i\binom ui\binom v{r-i}=\binom kr$, followed by \cref{lem:restricted-sum}, rewrites \eqref{eq:formal-harmonic} as
\begin{align}\label{eq:harmonic-binomial-direct}
 \sum_{n=1}^{k-2}
 \left(\binom nr+\binom ns-\frac12\binom kr\right)D_n
 =-\zzf r\zzf s.
\end{align}
Since $\binom{k-1}r+\binom{k-1}{s}=\binom kr$, comparing with \eqref{eq:direct-beta-certificate} gives $y_r+y_{k-r}=\zzf r\zzf{k-r}$ for $2\leq r\leq k-2$. At the boundary, \cref{lem:restricted-sum} and reversal give
\[
 \sum_{n=1}^{k-2}\left(n-\frac{k-1}{2}\right)D_n=-\mu_k,
\]
which together with the first sum in \eqref{eq:boundary-sums}, substituted in \eqref{eq:direct-beta-certificate}, gives
\[
 y_1=0,\qquad y_{k-1}=-\mu_k.
\]
Thus the stuffle and shuffle values agree in every case. Sending each product symbol to this common value defines a realization of $\Dk$. By construction $Z_k\mapsto0$, and the boundary calculation gives $Z_{1,k-1}\mapsto y_1=0$. Each even product symbol $P_{r,k-r}$ maps to $\zzf r\zzf{k-r}=0$, so the realization descends to $\Dtk$.

Moreover,
\[
 P_{r,k-r}\longmapsto\zzf r\zzf{k-r}
 \qquad(2\leq r\leq k-2).
\]
By \eqref{eq:four-one-product},
\begin{align}\label{eq:four-one-bridge}
 2\zaf{r-1,1,s-1,1}={}&rs\zzf r\zzf s\notag\\
 &+\zaf{1,r-1,s}-\zaf{1,s-1,r}.
\end{align}
Substituting \eqref{eq:direct-beta-certificate} in \eqref{eq:H-source-element}, using $y_r+y_s=\zzf r\zzf s$, and applying \eqref{eq:formal-Hoffman-depth-three} with $(a,b,c)=(1,r-1,s)$ and $(1,s-1,r)$ gives
\[
 H_{r,s}\longmapsto\frac12\left(
 rs\zzf r\zzf s+\zaf{1,r-1,s}-\zaf{1,s-1,r}
 \right).
\]
By \eqref{eq:four-one-bridge}, this is $\zaf{r-1,1,s-1,1}$. Thus the descended map is the map prescribed in \eqref{eq:direct-beta-rules}.

It remains to determine its image. Add $\mu_k$ back to the first coordinate in \eqref{eq:direct-beta-certificate}, and denote the resulting coordinates by $x_r$. Since $x_{k-1}=-\mu_k$, the spans of the $x_r$ and the $y_r$ agree. For $1\leq r\leq k-2$, put
\[
 U_r=(-1)^rx_r+\binom{k-1}{r}x_{k-1}.
\]
Then
\[
 U_r=\sum_{n=1}^{k-2}\binom nr
 \left(\zaf{n,1,k-n-1}-\zzf n\zzf{k-n}\right),
\]
and binomial inversion gives
\[
 \zaf{n,1,k-n-1}-\zzf n\zzf{k-n}
 =\sum_{r=n}^{k-2}(-1)^{r-n}\binom rnU_r.
\]
Thus every difference on the left belongs to the image. For $2\leq n\leq k-2$, \eqref{eq:harmonic-binomial-direct} gives $\zzf n\zzf{k-n}=x_n+x_{k-n}$. Hence $\zaf{n,1,k-n-1}$ belongs to the image. This is also true for $n=1$, since $\zzf 1=0$. Finally, \eqref{eq:formal-Hoffman-depth-three} expresses every depth-three value in terms of these values and the products $\zzf n\zzf{k-n}$. The image therefore contains $\Fil_3\ZAfw{k}$. Conversely, the image is spanned by the values $\zaf{r-1,1,s-1,1}$, which belong to $\Fil_3\ZAfw{k}$ by \cref{cor:explicit-true-single-product}. The image is therefore $\Fil_3\ZAfw{k}$.
\end{proof}

In the following, we want to show that $\beta_k^{\mathrm f}$ is also injective. For this, we use motivic multiple zeta values. Let $\mathcal H$ be the algebra of motivic multiple zeta values and put $\overline{\mathcal H}=\mathcal H/(\zeta^{\mathfrak m}(2))$. By $D_d\overline{\mathcal H}$ we denote the usual depth filtration, i.e. a motivic value with $d$ entries has depth at most $d$. Since motivic multiple zeta values satisfy the regularized double shuffle relations (see the introduction of~\cite{Y}), there is a homomorphism $\Zbarf\longrightarrow\overline{\mathcal H}$. Composing it with $\varphi$, we obtain
\begin{align}\label{eq:motivic-finite-realization}
 \begin{aligned}
 \rho:\ZAf&\longrightarrow\overline{\mathcal H},\\
 \zaf{k_1,\ldots,k_r}&\longmapsto\zeta_{\mathcal S}^{\mathfrak m}(k_1,\ldots,k_r),
 \end{aligned}
\end{align}
where the expression on the right is defined by the same regularized sum as $\zsf{k_1,\ldots,k_r}$, with formal values replaced by motivic values. In particular, $\rho(\zzf{k})=\zeta^{\mathfrak m}(k)$ by \cref{cor:true-single}.

\begin{prop}\label{prop:motivic-depth-four-bound}
For every even $k\geq4$, one has
\[
 \rho(\Fil_4\ZAfw{k})=D_2\overline{\mathcal H}_k,
 \qquad
 \dim_{\QQ}\Fil_4\ZAfw{k}\geq\frac{k}{2}-2-\dim_{\QQ}\mathcal S_k\,.
\]
\end{prop}
\begin{proof}
Let
\[
 P_k=\Span_{\QQ}\{\zeta^{\mathfrak m}(r)\zeta^{\mathfrak m}(s):r,s\geq3\text{ odd},\ r+s=k\}\subset D_2\overline{\mathcal H}_k\,.
\]
Since $\zzf{r}$ has depth at most two, $P_k$ is contained in $\rho(\Fil_4\ZAfw{k})$. Notice that $D_1\overline{\mathcal H}_k=0$ for even $k>0$, and that $D_2\overline{\mathcal H}_k=\QQ\zeta^{\mathfrak m}(k)$ for odd $k>0$, where $\zeta^{\mathfrak m}(1)=0$. The second assertion follows from the motivic realization of \eqref{eq:formal-double-zeta-odd}.

We apply Proposition~3.1 of~\cite{Y} in depth three. Its motivic version is stated in Remark~3.2(1), and Remark~3.2(4) gives the congruence modulo terms of depth at most one and products whose total depth is at most three. The proof uses regularized double shuffle relations and keeps these depth bounds, so this refinement also holds for motivic values. In even weight the error terms belong to $P_k$ by the preceding paragraph. Yasuda works with a variant of the symmetric values, which by Lemma~2.1 of~\cite{Y} differs from the usual symmetric values only by symmetric values of depth at most two. Those terms vanish in even weight by \cref{prop:low-depth}. For $a,b,c\geq1$ with $a+b+c=k$, we obtain
\begin{align}\label{eq:motivic-yasuda-depth-three}
 \rho(\zaf{c,b,a})&\equiv W_{a,b,c}\pmod{P_k},\\
 W_{a,b,c}&=-(-1)^a\sum_{\substack{i,j\geq0\\i+j=a}}\binom{b+i-1}{i}\binom{c+j-1}{j}\zeta^{\mathfrak m}(c+j,b+i)\notag\\
 &\hspace{1.8em}+(-1)^c\sum_{\substack{i,j\geq0\\i+j=c}}\binom{a+i-1}{i}\binom{b+j-1}{j}\zeta^{\mathfrak m}(b+j,a+i)\,.
 \notag
\end{align}
All motivic values in this formula are taken in $\overline{\mathcal H}$, with non-admissible values regularized at $T=0$.

Next, we show that the elements $W_{a,b,c}$ span $D_2\overline{\mathcal H}_k/P_k$. Let $\lambda:D_2\overline{\mathcal H}_k\longrightarrow\QQ$ vanish on $P_k$ and on every $W_{a,b,c}$, and put
\[
 f(X,Y)=\sum_{\substack{r,s\geq1\\r+s=k}}\lambda\bigl(\zeta^{\mathfrak m}(s,r)\bigr)X^{r-1}Y^{s-1}\,.
\]
The regularized stuffle and shuffle relations give
\[
 f(X,Y)+f(Y,X)=0,
 \qquad
 f(X,X+Y)+f(Y,X+Y)=0\,.
\]
Here the single terms vanish because $k$ is even, and the products of single values belong to $P_k$. Define the polynomial
\[
 g(X,Y,Z)=\frac{f(Y-X,Z-X)-f(Y,Z)}{X}\,.
\]
The coefficient of $X^{a-1}Y^{b-1}Z^{c-1}$ in $-g(X,Y,Z)+g(Z,X,Y)$ is $\lambda(W_{a,b,c})$. Hence $g(X,Y,Z)=g(Z,X,Y)$. Since $f$ has even degree, the two double shuffle identities give $f(Y-X,-X)=f(X,Y)$. Setting $Z=0$ in the cyclic identity now gives
\[
 f(X,Y)+X(\partial_X+\partial_Y)f(X,Y)=f(Y,0)\,.
\]
At $X=0$, antisymmetry implies $f(Y,0)=0$. Interchanging $X$ and $Y$ then gives $(X-Y)(\partial_X+\partial_Y)f=0$, and therefore $f=0$. Thus $\lambda=0$, which proves the spanning assertion.

It follows from \eqref{eq:motivic-yasuda-depth-three} that the elements $\rho(\zaf{c,b,a})$ together with $P_k$ span $D_2\overline{\mathcal H}_k$. Both are contained in $\rho(\Fil_4\ZAfw{k})$. Conversely, strong parity gives $\Fil_4\ZAfw{k}=\Fil_3\ZAfw{k}$, and \eqref{eq:motivic-yasuda-depth-three} shows that this image is contained in $D_2\overline{\mathcal H}_k$. This proves the equality of spaces. Finally, the exact sequence in \cite[Theorem~1.6(i)]{L} gives
\[
 \sum_{k>0}\dim_{\QQ}\operatorname{gr}^{D}_2\overline{\mathcal H}_k\,t^k
 =\left(\frac{t^3}{1-t^2}\right)^2-\frac{t^{12}}{(1-t^4)(1-t^6)}\,.
\]
Since $D_1\overline{\mathcal H}_k=0$ for even $k>0$, the coefficient of $t^k$ is $k/2-2-\dim_{\QQ}\mathcal S_k$, which gives the lower bound.
\end{proof}

\begin{proof}[Proof of \cref{main:double-zeta-surjection}]
Part~(i) is \cref{prop:H-spans-Dtilde}. For part~(ii), \cref{lem:direct-beta-certificate} shows that the rule in the statement is well defined and has image $\Fil_3\ZAfw{k}=\Fil_4\ZAfw{k}$. By \cref{prop:dimension-Dtilde,prop:motivic-depth-four-bound}, the dimension of the source equals the lower bound for the target. Therefore $\beta_k^{\mathrm f}$ is an isomorphism. Since the $H_{r,s}$ span the source, the target is also spanned by the values $\zaf{r-1,1,s-1,1}$.
\end{proof}

\begin{cor}\label{cor:filtered-formal-kz-injective}
For every even $k\geq4$, the map $\rho$ induces an isomorphism
\[
 \Fil_4\ZAfw{k}\simeq D_2\overline{\mathcal H}_k\,.
\]
In particular, $\varphi$ is injective on $\Fil_4\ZAfw{k}$.
\end{cor}
\begin{proof}
By \cref{main:double-zeta-surjection,prop:motivic-depth-four-bound}, the restriction of $\rho$ is a surjection between spaces of the same dimension. Hence it is injective. Since $\rho$ factors through $\varphi$, the restriction of $\varphi$ is injective as well.
\end{proof}

\begin{prop}\label{prop:true-single-algebra-independent}
The elements $\zzf{3},\zzf{5},\zzf{7},\ldots$ are algebraically independent, and $\varphi$ is injective on the subalgebra they generate.
\end{prop}
\begin{proof}
Brown's description of motivic multiple zeta values identifies $\overline{\mathcal H}$ with the shuffle algebra on letters $f_3,f_5,f_7,\ldots$, with $\zeta^{\mathfrak m}(2r+1)$ mapped to $f_{2r+1}$ (see \cite[Theorems~3.1 and~3.3]{B}). A shuffle monomial in these letters is a nonzero multiple of the sum of all words with the corresponding multiset of letters. Different monomials have disjoint supports, so the letters are algebraically independent. Since $\rho(\zzf{2r+1})=\zeta^{\mathfrak m}(2r+1)$, the assertion follows.
\end{proof}

\begin{proof}[Proof of \cref{main:classical-beta}]
The map $\beta_k=\pi_{\mathcal A}\circ\beta_k^{\mathrm f}$ is surjective. Since $\dim_{\QQ}\Dtk=k/2-2-\dim_{\QQ}\mathcal S_k$ and there are $k/2-2$ elements $H_{r,s}$, their images satisfy at least $\dim_{\QQ}\mathcal S_k$ independent relations.
\end{proof}

\begin{rem}\label{rem:kina-comparison}
For $2\leq r\leq k-2$, the finite realization of $y_r$ agrees, after reversal, with the element $Z_{\mathcal A}(k-r,r)$ of Kina, and the evaluation of the shuffle side above corresponds to his shuffle relation~\cite{Ki}. Zhao proved the corresponding finite stuffle relation using the finite harmonic identity from the earlier version of this paper~\cite[Proposition~5.6(i)]{Z}. Here the harmonic identity of \cref{prop:harmonic-from-hoffman} yields the stuffle evaluation $y_r+y_{k-r}=\zzf r\zzf{k-r}$, which makes the assignment a realization of $\Dk$.
\end{rem}

By \cref{main:classical-beta}, \cref{main:double-zeta} is equivalent to
\begin{align}\label{eq:sharp-depth-four-bound}
 \dim_{\QQ}\Fil_4\ZAw{k}\geq\frac{k}{2}-2-\dim_{\QQ}\mathcal S_k\,.
\end{align}
The motivic argument gives this lower bound for $\dim_{\QQ}\Fil_4\ZAfw{k}$, but not for the finite realization.

\subsection{From period polynomials to relations}
\label{subsec:recursive-period}

In this subsection, we explain how to obtain from an even period polynomial $p\in W_k^{\even}$ an explicit relation among the values $\zaf{r-1,1,s-1,1}$. The construction consists of three steps, and we give its proof in \Cref{subsec:explicit-period} after two examples. Here $[t^n]f(t)$ denotes the coefficient of $t^n$ in a polynomial or power series $f(t)$.

Suppose first that $k\geq8$. Put $m=k-2$, $d=\lfloor m/6\rfloor$ and
\[
 I=\{3,5,\ldots,k-3\},\qquad I^-=\{r\in I:2r<k\}.
\]
\begin{enumerate}[(i)]
\item Write
\[
 p(X,Y)=\sum_{r+s=k}\binom m{r-1}p_{r,s}X^{r-1}Y^{s-1}
\]
and start with
\begin{align}\label{eq:recursive-transfer-seed}
 c_r=\begin{cases}
 p_{r+2,k-r-2},&2d+1\leq r<k/2-2,\\
 0,&\text{otherwise}
 \end{cases}
 \qquad(r\in I^-).
\end{align}
The following fixed polynomials give a triangular reduction of these coefficients:
\begin{align}\label{eq:recursive-transfer-polynomials}
 e_b(t)&=t(t-1)(t+1)(2t-1)(t-2)
 (t^2-t+1)^{(m-6-6b)/2}\bigl(t^2(t-1)^2\bigr)^b,\notag\\
 o_b(t)&=\bigl((m+1)t^2+mt+1\bigr)e_b(t)+(t-t^3)e_b'(t)
 \qquad(0\leq b<d).
\end{align}
For $b=d-1,d-2,\ldots,1$, successively replace $c_{2b+1}$ by
\begin{align}\label{eq:recursive-transfer-elimination}
 c_{2b+1}\longleftarrow c_{2b+1}
 +\frac{1}{4(b+1)}\sum_{r\in I^-}c_r[t^r]o_b(t).
\end{align}
Then replace all coefficients simultaneously by
\begin{align}\label{eq:recursive-transfer-last-step}
 c_r\longleftarrow c_r+\frac{2r-k}{6m}
 \sum_{u\in I^-}c_u[t^u]o_0(t)
 \qquad(r\in I^-).
\end{align}
Extend the result by $c_{k-r}=-c_r$ and, if $k/2\in I$, by $c_{k/2}=0$.

\item Form the two homogeneous polynomials
\begin{align}\label{eq:clean-transfer-HF}
 H(X,Y)&=\sum_{r\in I}\binom kr c_rX^rY^{k-r},\notag\\
 F(X,Y)&=m![t^m]\left(e^{Yt}\frac{Xt}{e^{Xt}-1}
 \sum_{r\in I}\frac{r+1}{r!}c_r(Xt)^{r-1}\right).
\end{align}
The quotient $Xt/(e^{Xt}-1)$ is expanded at $t=0$. Let $\Pi_m$ retain only the monomials $X^aY^{m-a}$ with $a$ even and $2\leq a\leq m-2$, and put
\begin{align}\label{eq:clean-transfer-Q}
 Q(X,Y)=\Pi_m\bigg(&\frac{\partial_X\partial_Y}{k(k-1)}
 \bigl(H(X+Y,Y)+H(X+Y,X)\notag\\
 &\hspace{8em}-X\partial_XH(X,Y)+H(X,Y)\bigr)\notag\\
 &+2\bigl(F(X+Y,X)-F(Y,X)-F(1,1)(X+Y)^m\bigr)\bigg).
\end{align}

\item Make the change of variables
\begin{align}\label{eq:clean-transfer-shear}
 Q(X+Y,Y)-Q(Y,X+Y)-Q(X,Y)
 =\sum_{r+s=k}\binom m{r-1}b_{r,s}X^{r-1}Y^{s-1}.
\end{align}
The coefficients of the relation are
\begin{align}\label{eq:clean-transfer-C}
 C_r(p)=c_r+\frac{b_{r,k-r}}{r(k-r)}\qquad(r\in I).
\end{align}
\end{enumerate}
For $k=4,6$, set $C_r(p)=0$.

\begin{thm}\label{thm:recursive-period-transfer}
For every even $k\geq4$ and $p\in W_k^{\even}$, one has
\begin{align}\label{eq:recursive-period-relation}
 \sum_{\substack{r+s=k\\r,s\geq3\ \odd}}
 C_r(p)\zaf{r-1,1,s-1,1}=0.
\end{align}
The map $p\mapsto(C_r(p))_r$ induces an isomorphism from $W_k^{\even}/\QQ(X^{k-2}-Y^{k-2})$ to the space of all relations among these formal finite values.
\end{thm}
These coefficients also give relations among finite multiple zeta values. Notice that the polynomials $e_b$ and $o_b$ in step~(i) depend only on the weight, and that steps~(ii) and~(iii) only use differentiation, substitution and coefficient extraction.

\begin{ex}\label{ex:weight-twelve-recursive-relation}
Let $k=12$ and $p=X^2Y^2(X^2-Y^2)^3$. Here $d=1$, $I^-=\{3,5\}$ and $([t^3]o_0,[t^5]o_0)=(6,12)$. The initial pair is $(1/70,0)$, and \eqref{eq:recursive-transfer-last-step} gives
\[
 (c_3,c_5)=\left(\frac1{175},-\frac1{350}\right).
\]
After antisymmetric extension, step~(ii) gives
\[
 Q(X,Y)=\frac{-945X^2Y^8+7903X^4Y^6+5875X^6Y^4-269X^8Y^2}{175}.
\]
Reading the coefficients in step~(iii), we obtain
\[
 (C_3(p),C_5(p),C_7(p),C_9(p))
 =\frac1{1575}(16,9,18,-2).
\]
Thus the relation is
\[
 16\za{2,1,8,1}+9\za{4,1,6,1}
 +18\za{6,1,4,1}-2\za{8,1,2,1}=0,
\]
and the same relation holds for formal finite multiple zeta values.
\end{ex}

\begin{ex}\label{ex:weight-sixteen-recursive-relation}
For $k=16$, take
\[
 p(X,Y)=X^2Y^2(X^2-Y^2)^3(2X^4-X^2Y^2+2Y^4).
\]
The construction gives
\[
 (C_3(p),C_5(p),C_7(p),C_9(p),C_{11}(p),C_{13}(p))
 =\frac{2}{2457}(366,120,59,53,111,327).
\]
Thus
\begin{align*}
 0={}&366\za{2,1,12,1}+120\za{4,1,10,1}
 +59\za{6,1,8,1}\\
 &+53\za{8,1,6,1}+111\za{10,1,4,1}
 +327\za{12,1,2,1}.
\end{align*}
This finite relation was given by Kaneko and Zagier~\cite{KZ}. The same relation holds for formal finite multiple zeta values, and since $\dim_{\QQ}\mathcal S_{16}=1$, it is, up to a scalar, the only relation among the formal values $\zaf{2a,1,2b,1}$ of weight $16$.
\end{ex}

\subsection{Proof of the coefficient formula}
\label{subsec:explicit-period}

The elements $H_{r,s}$ can be written explicitly in terms of the odd--odd classes. For odd $m$ with $1\leq m\leq k-3$ and $r+s=k$, set, following~\cite{GKZ},
\begin{align}\label{eq:lambda-prime}
 \lambda'_{m,k-m}(r,s)
 =\sum_{\ell=0}^{r-2}
 \binom{k-2-\ell}{m-1}\binom{r-1}{\ell}B_{k-m-\ell-1},
\end{align}
where $B_\nu$ denotes the $\nu$th Bernoulli number with $B_1=-\frac12$, and $B_\nu=0$ for $\nu<0$. Since $Z_k=0$ in $\Dtk$, the identity \cite[(7)]{GKZ} becomes
\begin{align}\label{eq:GKZ-seven-reduced}
 Z_{m+1,k-m-1}
 =-\frac{2}{m}
 \sum_{\substack{r+s=k\\r,s\ \odd}}
 \lambda'_{m,k-m}(r,s)\,Z_{r,s}
 \qquad\text{in }\Dtk,
\end{align}
where $Z_{1,k-1}=0$ and, by the sum formula from the proof of \cref{prop:period-polynomial-relations}, $Z_{k-1,1}=-\sum_{3\leq a\leq k-3,\ a\ \odd}Z_{a,k-a}$.

\begin{prop}\label{prop:H-in-Z}
Let $k\geq4$ be even and let $r,s\geq3$ be odd with $r+s=k$. Then, in $\Dtk$,
\begin{align}\label{eq:H-in-Z}
 2H_{r,s}
 ={}&(rs-r+1)Z_{r,s}+(rs+s-1)Z_{s,r}
 +\sum_{\substack{3\leq\ell\leq k-3\\\ell\ \odd}}
 \left(\binom{k-\ell}{s}-\binom{k-\ell}{r}\right)P_{\ell,k-\ell}
 \notag\\
 &+2\sum_{\substack{r'+s'=k\\r',s'\ \odd}}
 \left(\frac{r+1}{r}\,\lambda'_{r,s}(r',s')
      -\frac{s+1}{s}\,\lambda'_{s,r}(r',s')\right)Z_{r',s'},
\end{align}
where $P_{a,b}=Z_{a,b}+Z_{b,a}$.
\end{prop}

\begin{proof}
Apply \eqref{eq:GKZ-seven-reduced} with $m=s$ and with $m=r$ to the two even--even terms $(s+1)Z_{s+1,r-1}$ and $-(r+1)Z_{r+1,s-1}$ in \eqref{eq:H-source-element}.
\end{proof}

\begin{proof}[Proof of \cref{thm:recursive-period-transfer}]
Recall that $m=k-2$. Let $\mathcal Q_k$ be the quotient of $\Dtk$ by the odd product symbols. Its dual space can be identified with the space of homogeneous polynomials $Z(X,Y)\in\QQ[X,Y]$ of degree $m$ satisfying \eqref{eq:dihedral-anti-invariance}. To see this, notice that all product symbols vanish in $\mathcal Q_k$, including $P_{1,k-1}$ by the sum formula, so that the two identities in \eqref{eq:dihedral-anti-invariance} are exactly the stuffle and shuffle relations. Conversely, a polynomial satisfying these identities is divisible by $J$, so its boundary coefficients vanish and it defines a linear functional on $\mathcal Q_k$. By \eqref{eq:dihedral-factorization}, this dual space has basis
\[
 E_b(X,Y)=J A_2^{(m-6-6b)/2}A_6^b\qquad(0\leq b<d).
\]
Write
\[
 P_{r,b}=[t^{r-1}]e_b(t),\qquad O_{r,b}=[t^r]o_b(t)
 \qquad(r\in I^-,\ 0\leq b<d).
\]
The functional corresponding to $E_b$ takes $Z_{r,k-r}$ to $P_{r,b}$ and $H_{r,k-r}$ to $-O_{r,b}/2$, as in the proof of \cref{prop:H-spans-Dtilde}.

The leading terms are
\[
 e_b(t)=-2t^{2b+1}+(m-1-2b)t^{2b+2}+\cdots,
 \qquad o_b(t)=-4(b+1)t^{2b+1}+\cdots.
\]
We also have the moment identity
\begin{align}\label{eq:recursive-transfer-moment}
 \sum_{r\in I^-}(2r-k)O_{r,b}=-6m\delta_{b,0}.
\end{align}
To see this, antisymmetry gives $[t^{k-r}]o_b=-[t^r]o_b$. The boundary coefficients are $[t]o_b=-4\delta_{b,0}$ and $[t^{k-1}]o_b=4\delta_{b,0}$, while the calculation in \cref{lem:dihedral-odd-coefficients} gives $o_b'(1)=-4m\delta_{b,0}$ and $o_b'(-1)=0$. Hence
\[
 \sum_{\substack{1\leq r\leq k-1\\r\ \odd}}
 (2r-k)[t^r]o_b(t)=-4m\delta_{b,0}.
\]
Removing the boundary terms and pairing $r$ with $k-r$ proves \eqref{eq:recursive-transfer-moment}.

By the leading term of $o_b$, each replacement in \eqref{eq:recursive-transfer-elimination} makes $\sum_rc_rO_{r,b}$ zero and preserves the sums for larger $b$. Formula \eqref{eq:recursive-transfer-moment} shows that the simultaneous replacement \eqref{eq:recursive-transfer-last-step} makes the remaining sum zero and preserves all the others. Thus the resulting vector lies in the left kernel of $O$, i.e.
\[
 \sum_{r\in I^-}c_rO_{r,b}=0\qquad(0\leq b<d).
\]
It follows that $\sum_{r\in I}c_rH_{r,k-r}$ vanishes in $\mathcal Q_k$.

A period polynomial $p\in W_k^{\even}$ is called restricted if $p(X,0)=0$. We next show that step~(i) identifies the restricted period polynomials with this left kernel. Subtracting a multiple of $X^m-Y^m$, we may assume that $p$ is restricted. By \cref{prop:period-polynomial-relations}, the vector $(p_{r,k-r})_{r\in I^-}$ lies in the left kernel of $P$, since $q_{r,s}-q_{s,r}=p_{r,s}$. Conversely, a relation in $\mathcal Q_k$ lifts to a relation in $\Dtk$ after subtracting odd products, whose coefficient vectors are symmetric. Taking the antisymmetric part and applying the same proposition gives a restricted period polynomial. Therefore its coefficients identify the restricted period space with the left kernel of $P$.

The rows $r=3,5,\ldots,2d+1$ of $P$ form a triangular square matrix with nonzero diagonal $m-1-2b$. Therefore a restricted period polynomial is determined by its remaining coefficients $p_{r,k-r}$ with $r\in I^-$ and $r\geq2d+3$, and these can be chosen freely. The shift in \eqref{eq:recursive-transfer-seed} identifies these free coefficients with vectors supported on $2d+1\leq r<k/2-2$. On the other hand, every correction in step~(i) lies in the span of
\[
 (2r-k)_{r\in I^-}
 \quad\text{and the coordinate vectors indexed by }3,5,\ldots,2d-1.
\]
This is a complementary space to the initial supported vectors. Its last coordinate determines the coefficient of $(2r-k)_r$, and the remaining low coordinates determine the other coefficients. The leading terms and \eqref{eq:recursive-transfer-moment} show that the restriction of $O$ to this space is invertible. Thus step~(i) is an isomorphism from the restricted period space to the left kernel of $O$.

It remains to verify the polynomial conversion in step~(ii). For any antisymmetric coefficients $c_{k-r}=-c_r$, the defining formula for $H_{r,s}$ gives
\begin{align}\label{eq:clean-transfer-source}
 \sum_{r\in I}c_rH_{r,k-r}
 ={}&-\sum_{r\in I}c_r\bigl((r-1)Z_{r,k-r}+(r+1)Z_{r+1,k-r-1}\bigr)\notag\\
 &-\sum_{r\in I}\mu_rP_{r,k-r},
 \qquad \mu_r=\sum_{u\in I}c_u\binom{k-r}u.
\end{align}
Apply \eqref{eq:GKZ-seven-reduced} and use
\[
 Z_{1,k-1}=0,\qquad Z_{k-1,1}=-\sum_{r\in I}Z_{r,k-r}.
\]
The resulting coefficient of $Z_{r,k-r}$ is
\begin{align}\label{eq:clean-transfer-delta}
 \delta_r={}&-(r-1)c_r-\mu_r-\mu_{k-r}\notag\\
 &+2\sum_{u\in I}\frac{u+1}{u}c_u
 \bigl(\lambda'_{u,k-u}(r,k-r)-\lambda'_{u,k-u}(k-1,1)\bigr).
\end{align}
We check that the polynomial $Q$ in step~(ii) is
\begin{align}\label{eq:clean-transfer-Q-coordinates}
 Q(X,Y)=\sum_{r\in I}\binom m{r-1}\delta_rX^{r-1}Y^{k-r-1}.
\end{align}
Write $F(X,Y)=\sum_{n=0}^m\binom mn f_nX^nY^{m-n}$. The Bernoulli generating series gives
\[
 f_n=\sum_{u\in I}\frac{u+1}{u}c_u\binom n{u-1}B_{n-u+1},
\]
where terms with $u-1>n$ are zero. The coefficient of $\binom m{r-1}X^{r-1}Y^{k-r-1}$ in $F(X+Y,X)-F(Y,X)$ is
\[
 \sum_{j=0}^{r-2}\binom{r-1}j f_{m-j}
 =\sum_{u\in I}\frac{u+1}{u}c_u\lambda'_{u,k-u}(r,k-r).
\]
At $r=k-1$, this equals $F(1,1)$ because $F(0,1)=0$. The Bernoulli part of \eqref{eq:clean-transfer-Q} therefore agrees with \eqref{eq:clean-transfer-delta}.

The binomial theorem and antisymmetry give
\[
 [X^rY^{k-r}]H(X+Y,Y)=-\binom kr\mu_r,
 \qquad r(k-r)\binom kr=k(k-1)\binom m{r-1}.
\]
Thus the two substituted terms inside the mixed derivative give $-\mu_r-\mu_{k-r}$, and the remaining terms give $-(r-1)c_r$. Applying $\Pi_m$ proves \eqref{eq:clean-transfer-Q-coordinates}, so that
\[
 \sum_{r\in I}c_rH_{r,k-r}=\sum_{r\in I}\delta_rZ_{r,k-r}
 \qquad\text{in }\Dtk.
\]
This element vanishes modulo odd products. The identification with the left kernel of $P$ shows that $Q(X,Y)-Q(Y,X)$ is a restricted period polynomial. Its relation in \cref{prop:period-polynomial-relations} is obtained from $Q(X+Y,Y)-Q(Y,X+Y)$. Therefore the coefficients $b_{r,s}$ of \eqref{eq:clean-transfer-shear} satisfy $b_{r,s}=b_{s,r}$ for odd $r,s$, and
\[
 \sum_{r\in I}\delta_rZ_{r,k-r}
 +\sum_{r\in I}b_{r,k-r}Z_{r,k-r}=0.
\]
The identity $H_{r,s}+H_{s,r}=rsP_{r,s}$ gives
\[
 \sum_{r\in I}\frac{b_{r,k-r}}{r(k-r)}H_{r,k-r}
 =\sum_{r\in I}b_{r,k-r}Z_{r,k-r},
\]
including the middle index, where $H_{r,r}=r^2P_{r,r}/2$. This proves $\sum_rC_r(p)H_{r,k-r}=0$ and hence \eqref{eq:recursive-period-relation} by \cref{main:double-zeta-surjection}.

Finally, $C_r(p)-C_{k-r}(p)=2c_r$ for $r\in I^-$, so the map is injective modulo $\QQ(X^m-Y^m)$. The source and the relation space both have dimension $\dim_{\QQ}\mathcal S_k$, by \cref{eq:H-relation-dimension,main:double-zeta-surjection}, which proves that it is an isomorphism. In the cases $k=4,6$, both spaces are zero. Applying $\pi_{\mathcal A}$ gives the finite relation.
\end{proof}

\subsection{Bernoulli coordinates and a basis conjecture}

By \cref{thm:recursive-period-transfer}, the construction above gives all relations among the values $\zaf{r-1,1,s-1,1}$ in every even weight. In this subsection, we give a determinant construction whose completeness follows if \cref{conj:detA} holds, and discuss the equivalent basis question in $\Dk$.

Substituting $Z_{1,k-1}=0$ and the sum formula for $Z_{k-1,1}$ in \eqref{eq:H-in-Z} expresses each $H_{r,s}$ as an explicit rational linear combination
\begin{align}\label{eq:H-in-Z-matrix}
 2H_{r,s}=\sum_{\substack{3\leq r'\leq k-3\\r'\ \odd}}
 A_{r,r'}\,Z_{r',k-r'}
\end{align}
of the classes $Z_{r',k-r'}$ with $r'$ odd and $3\leq r'\leq k-3$. We write $A_k=(A_{r,r'})$ for the resulting square matrix, whose rows and columns are indexed by the odd integers $3\leq r,r'\leq k-3$.

Since the $H_{r,s}$ satisfy relations in $\Dtk$ whenever $\mathcal S_k\neq0$, the expansion \eqref{eq:H-in-Z-matrix} is not unique in $\Dtk$. It becomes unique after passing to the quotient
\begin{align}\label{eq:Dbar}
 \Dbk=\Dk/\bigl(\QQ Z_k+\QQ Z_{1,k-1}\bigr),
\end{align}
of which $\Dtk$ is the further quotient by the even product symbols. As before, we read the right-hand side of \eqref{eq:H-source-element} in $\Dk$, with $P_{a,b}=Z_{a,b}+Z_{b,a}+Z_k$, so that the $H_{r,s}$ make sense in every quotient of $\Dk$.

\begin{lem}\label{lem:Dbar-basis}
Let $k\geq6$ be even.
\begin{enumerate}[(i)]
\item The classes $Z_{r,k-r}$ with $r$ odd and $3\leq r\leq k-3$ form a basis of $\Dbk$.
\item The identity \eqref{eq:H-in-Z-matrix} holds in $\Dbk$, and $A_k$ is the unique matrix with this property.
\end{enumerate}
\end{lem}

\begin{proof}
The elements $Z_k$ and $Z_{1,k-1}$ are linearly independent in $\Dk$, as in the proof of \cref{prop:dimension-Dtilde}, so $\dim_{\QQ}\Dbk=\frac k2-2$. The odd--odd classes $Z_{a,k-a}$ with $a$ odd form a basis of $\Dk$~\cite[Theorem~2]{GKZ}. The relation \eqref{eq:formal-double-shuffle} for $\{r,s\}=\{1,k-1\}$ gives the sum formula
\[
 \sum_{a=1}^{k-1}Z_{a,k-a}=Z_{1,k-1}+Z_k
 \qquad\text{in }\Dk,
\]
while the even sum formula of~\cite[Theorem~1]{GKZ} gives $\sum_{a\ \even}Z_{a,k-a}=\frac34Z_k$. Subtracting, the odd part $\sum_{a\ \odd}Z_{a,k-a}$ lies in $\QQ Z_k+\QQ Z_{1,k-1}$, so that
\begin{align}\label{eq:odd-sum-Dbar}
 Z_{k-1,1}=-\sum_{\substack{3\leq r\leq k-3\\r\ \odd}}Z_{r,k-r}
 \qquad\text{in }\Dbk.
\end{align}
Hence the $\frac k2-2$ classes $Z_{r,k-r}$ with $r$ odd and $3\leq r\leq k-3$ span $\Dbk$, and they form a basis by dimension. This proves part~(i), and with it the uniqueness in part~(ii). For the identity, note that the identity \cite[(7)]{GKZ} holds in $\Dk$ up to a rational multiple of $Z_k$, so that \eqref{eq:GKZ-seven-reduced}, and with it \eqref{eq:H-in-Z}, hold already in $\Dbk$. Substituting $Z_{1,k-1}=0$ and \eqref{eq:odd-sum-Dbar} gives \eqref{eq:H-in-Z-matrix} in $\Dbk$.
\end{proof}

\begin{prop}\label{prop:Dk-basis}
Let $k\geq6$ be even. The following are equivalent:
\begin{enumerate}[(i)]
\item the elements
\[
 H_{r,s}\quad(r,s\geq3\ \odd,\ r+s=k),
 \qquad Z_k,\qquad Z_{1,k-1}
\]
form a basis of $\Dk$,
\item $\det A_k\neq0$.
\end{enumerate}
Both statements hold whenever $\mathcal S_k=0$.
\end{prop}

\begin{proof}
There are $\frac k2$ elements listed in (i), and $\dim_{\QQ}\Dk=\frac k2$ by \cite[Theorem~2]{GKZ}. Since $Z_k$ and $Z_{1,k-1}$ are linearly independent and span the kernel of the projection $\Dk\to\Dbk$, the elements in (i) form a basis of $\Dk$ if and only if the classes of the $H_{r,s}$ form a basis of $\Dbk$. By \cref{lem:Dbar-basis}, the latter holds if and only if $\det A_k\neq0$.

Suppose that $\mathcal S_k=0$. Then $\dim_{\QQ}\Dtk=\frac k2-2=\dim_{\QQ}\Dbk$ by \cref{prop:dimension-Dtilde}, so the surjection $\Dbk\to\Dtk$ is an isomorphism. The classes of the $H_{r,s}$ span $\Dtk$ by \cref{prop:H-spans-Dtilde}, hence they span $\Dbk$, and therefore $\det A_k\neq0$.
\end{proof}

\begin{conj}\label{conj:detA}
For every even $k\geq6$ one has $\det A_k\neq0$. Equivalently, the elements $H_{r,s}$ together with $Z_k$ and $Z_{1,k-1}$ form a basis of $\Dk$.
\end{conj}

By \cref{prop:Dk-basis}, the content of \cref{conj:detA} lies in the cusp weights, where it is the statement that every relation of \cref{prop:period-polynomial-relations} is realized coefficientwise by \eqref{eq:H-in-Z-matrix}. In each weight it is one explicit determinant condition with binomial and Bernoulli number entries, and we confirmed it by exact computation for every even $k\leq100$. Whenever it holds, every class $Z_{r,s}$ has a unique expression in the elements $H_{r,s}$, $Z_k$ and $Z_{1,k-1}$. In particular
\[
 Z_{r,k-r}=2\sum_{\substack{3\leq r'\leq k-3\\r'\ \odd}}
 (A_k^{-1})_{r,r'}\,H_{r',k-r'}
 \qquad\text{in }\Dbk,
\]
and expressions in $\Dtk$ are obtained by projection. We do not know a closed formula for $A_k^{-1}$, a difficulty already encountered in~\cite[Section~3]{KZ} for a related triangular system. For $p\in W_k^{\even}$, write
\[
 p(X+Y,Y)=\sum_{r+s=k}\binom{k-2}{r-1}q_{r,s}X^{r-1}Y^{s-1}.
\]
For odd $r$ with $3\leq r\leq k-3$, let $\widehat c_r(p)$ be the determinant of the matrix obtained from $A_k$ by replacing the row indexed by $r$ by
\[
 \bigl(2(q_{r',k-r'}-q_{k-1,1})\bigr)_{r'} ,
\]
where $r'$ runs over the odd integers from $3$ to $k-3$.

\begin{thm}\label{thm:cofactor-relations}
Let $k\geq6$ be even and $p\in W_k^{\even}$. Then
\begin{align}\label{eq:cofactor-relation}
 \sum_{\substack{3\leq r\leq k-3\\r\ \mathrm{odd}}}
 \widehat c_r(p)H_{r,k-r}=0
 \qquad\text{in }\Dtk.
\end{align}
Consequently,
\begin{align}\label{eq:cofactor-finite-relation}
 \sum_{\substack{3\leq r\leq k-3\\r\ \mathrm{odd}}}
 \widehat c_r(p)\za{r-1,1,k-r-1,1}=0
\end{align}
holds for finite multiple zeta values. The same relation holds for formal finite multiple zeta values.

If $\det A_k\neq0$, these relations, as $p$ varies, span the space of all relations among the $H_{r,s}$, and they vanish identically only for $p\in\QQ(X^{k-2}-Y^{k-2})$.
\end{thm}

\begin{proof}
All indices below are odd and lie between $3$ and $k-3$. Let $C_{r,r'}$ be the cofactor of $A_k$ at $(r,r')$. Expansion along the replaced row gives
\[
 \widehat c_r(p)=2\sum_{r'}(q_{r',k-r'}-q_{k-1,1})C_{r,r'}.
\]
The adjugate identity therefore yields
\[
 \sum_r\widehat c_r(p)A_{r,r'}
 =2\det(A_k)(q_{r',k-r'}-q_{k-1,1}).
\]
Using \eqref{eq:H-in-Z-matrix}, we obtain in $\Dbk$
\[
 \sum_r\widehat c_r(p)\,2H_{r,k-r}
 =2\det(A_k)\sum_{r'}
 (q_{r',k-r'}-q_{k-1,1})Z_{r',k-r'}.
\]
The right-hand side vanishes after projection to $\Dtk$ by \cref{prop:period-polynomial-relations}, proving \eqref{eq:cofactor-relation}. Applying the unconditional map $\beta_k$ gives \eqref{eq:cofactor-finite-relation}, and applying $\beta_k^{\mathrm f}$ gives the formal relation.

If $\det A_k\neq0$, writing $b_{r'}=q_{r',k-r'}-q_{k-1,1}$, the coefficient row is
\[
 (\widehat c_r(p))_r=2\det(A_k)\,(b_{r'})_{r'}A_k^{-1}.
\]
The final assertions now follow from \cref{prop:period-polynomial-relations} and \eqref{eq:H-relation-dimension}.
\end{proof}

\begin{ex}\label{ex:weight-twelve-transfer-relation}
Let $k=12$ and $p=X^2Y^2(X^2-Y^2)^3$. One has
\[
 A_{12}=
 \begin{pmatrix}
  -39&-\frac{23}{3}&-\frac{35}{3}&-\frac{289}{9}\\[2pt]
  -90&-\frac{278}{5}&-\frac{310}{7}&-\frac{1594}{15}\\[2pt]
  90&\frac{628}{5}&\frac{800}{7}&\frac{1594}{15}\\[2pt]
  93&\frac{23}{3}&\frac{35}{3}&\frac{775}{9}
 \end{pmatrix},
 \qquad \det A_{12}=48672.
\]
The determinants in \cref{thm:cofactor-relations} are
\[
 (\widehat c_3(p),\widehat c_5(p),\widehat c_7(p),\widehat c_9(p))
 =(128,72,144,-16).
\]
After dividing by $8$, \eqref{eq:cofactor-finite-relation} becomes \eqref{eq:intro-weight-twelve-relation}. The same relation holds in the formal finite algebra by \cref{main:double-zeta-surjection}.
\end{ex}

\vspace{0.5cm}
\noindent{\bf AI \& computational resource disclosure:} The main results, their formulation, and the underlying idea of proof are the authors' own. ChatGPT and Claude were used throughout as research assistants: they carried out and checked computations, and a number of intermediate steps in the proofs were worked out in dialogue with these systems. Symbolic computations were performed in SageMath and Python. All statements and proofs were verified by the authors, who take sole responsibility for them.

\end{document}